\documentclass[final]{amsart}
\usepackage{amsthm}
\usepackage{amsmath}
\usepackage{amsfonts}
\usepackage{amssymb}
\usepackage[usenames]{color}
\usepackage[mathscr]{euscript}
\usepackage[hidelinks]{hyperref}
\def\natural{{\mathbb N}}
\def\zahlen{{\mathbb Z}}
\def\real{{\mathbb{R}}}
\def\rational{{\mathbb Q}} 
\def\ball#1,#2.{B(#1,#2)} 
\def\biglip{\text{\normalfont Lip}} 
\def\smllip{\text{\normalfont lip}} 
\def\clball#1,#2.{\bar B(#1,#2)} 
\def\sball#1,#2,#3.{B_{#1}(#2,#3)} 
\def\scball#1,#2,#3.{\bar B_{#1}(#2,#3)} 
\def\glip#1.{{\bf L}(#1)} 
\def\glipdec#1,#2.{{\bf L}_{#2}(#1)}
\def\pcreatedst#1#2{\expandafter\def\csname #1dst\endcsname##1,##2.{#2(##1,##2)}
\expandafter\def\csname #1dstp\endcsname##1.{{#2}_{##1}}
\expandafter\def\csname #1dname\endcsname{#2}} 
\def\pcreatenrm#1#2#3{\expandafter\def\csname
  #1nrm\endcsname##1.{\left #3 ##1\right #3_{#2}} \expandafter\def\csname
  #1nrmname\endcsname{\left #3\,\cdot\,\right #3_{#2}}} 
\def\pcreateasymnrm#1#2#3#4{\expandafter\def\csname
  #1nrm\endcsname##1.{\left #3 ##1\right #4_{#2}} \expandafter\def\csname
  #1nrmname\endcsname{\left #3\,\cdot\,\right #4_{#2}}} 
\def\wder#1.{{\mathscr{X}}({#1})}
\def\wform#1.{{\mathscr{E}}({#1})}
\newcommand{\dist}{\operatorname{dist}}
\def\mpush#1.{{#1}_{\sharp}} 
\newcommand{\tang}{\operatorname{Bw-up}}
\def\lipalg#1.{{\rm Lip}_{\text{\normalfont b}}(#1)} 
\def\albrep#1.{{\mathcal A}_{#1}} 
\def\ccompact#1.{{C_c(#1)}}
\newcommand{\diam}{\operatorname{diam}}
\newcommand{\on}{\:\mbox{\rule{0.1ex}{1.2ex}\rule{1.1ex}{0.1ex}}\:}
\DeclareMathOperator\frags{Frag} 
\DeclareMathOperator\rescf{Res} 
\DeclareMathOperator\dom{dom} 
\def\lebmeas#1.{\setbox1=\hbox{$#1$\unskip}{\mathcal L}^{\ifdim\wd1>0pt
    #1 \else 1 \fi}} 
\DeclareMathOperator\metdiff{md} 
\DeclareMathOperator\curves{Curves}
\DeclareMathOperator\acspace{AC}
\DeclareMathOperator\spt{spt} 
\def\locnorm#1,#2.{\left|{#1}\right|_{{#2},\text{\normalfont loc}}}
\def\lipfun#1.{{\rm Lip}(#1)} 
\def\pcreateconv#1#2{\expandafter\def\csname
  #1cvj\endcsname{\xrightarrow{#2}}} 
\pcreateconv{wk}{\text{\normalfont w*}}
\def\makeunderscoreletter{\catcode`_=11}
\def\makecolonletter{\catcode`:=11}
\def\unmakecolonletter{\catcode`:=12}
\def\makeunderscoresub{\catcode`_=8}
\def\sync #1.{\makeunderscoreletter\makecolonletter\expandafter\ifx\csname
  sync#1\endcsname\relax Undefined\else \csname
  sync#1\endcsname\fi\makeunderscoresub\unmakecolonletter}
\makeunderscoreletter\makecolonletter\def\syncrem:derivation_extension{
  2.115}\makeunderscoresub\unmakecolonletter
\makeunderscoreletter\makecolonletter\def\synccor:derbound{ 5.136}\makeunderscoresub\unmakecolonletter
\makeunderscoreletter\makecolonletter\def\syncthm:alb_derivation{ 3.11}\makeunderscoresub\unmakecolonletter
\makeunderscoreletter\makecolonletter\def\syncsec_7_cks{ 7}\makeunderscoresub\unmakecolonletter
\makeunderscoreletter\makecolonletter\def\synceq:cks_reg_dec{(7.34)}\makeunderscoresub\unmakecolonletter
\makeunderscoreletter\makecolonletter\def\synclem:cks_reg_dec_borel{ 7.35}\makeunderscoresub\unmakecolonletter
\makeunderscoreletter\makecolonletter\def\synceq:alberti_blow_up_psix{(7.55)}\makeunderscoresub\unmakecolonletter
\makeunderscoreletter\makecolonletter\def\synceq:cks_alberti_blow_up_esttwo{(7.60)}\makeunderscoresub\unmakecolonletter
\makeunderscoreletter\makecolonletter\def\synceq:cks_alberti_blow_up_estthree{(7.66)}\makeunderscoresub\unmakecolonletter
\makeunderscoreletter\makecolonletter\def\synceq:cks_alberti_blow_up_pelevenplusone{(7.71)}\makeunderscoresub\unmakecolonletter
\makeunderscoreletter\makecolonletter\def\synclem:local_alberti_blow_up{
   7.78}\makeunderscoresub\unmakecolonletter
\makeunderscoreletter\makecolonletter\def\synclem:der_alb_gap_flat_func{
  4.17}\makeunderscoresub\unmakecolonletter
\numberwithin{equation}{section} 
\theoremstyle{plain}
\newtheorem{lem}[equation]{Lemma}

\newtheorem{thm}[equation]{Theorem}
\newtheorem{cor}[equation]{Corollary}
\theoremstyle{definition}
\newtheorem{defn}[equation]{Definition}
\theoremstyle{remark}

\newtheorem{rem}[equation]{Remark}
\begin{document}
\title{The L\MakeLowercase{ip}-\MakeLowercase{lip} equality is stable under blow-up}
\author{Andrea Schioppa}
\address{ETHZ}
\email{andrea.schioppa@math.ethz.ch}
\thanks{The author was supported by the ``ETH Zurich Postdoctoral Fellowship Program and the Marie Curie Actions
  for People COFUND Program'' and the European Research Council grant
  n.291497}
\begin{abstract}
  We show that at generic points blow-ups/tangents of
  differentiability spaces are still differentiability spaces; this
  implies that an analytic condition introduced by Keith as an
  inequality (and later proved to actually be an equality) passes to
  tangents. As an application, we characterize the $p$-weak gradient
  on iterated blow-ups of differentiability spaces.
\end{abstract}
\maketitle
\tableofcontents
\section{Introduction}
\label{sec:intro}
\subsection*{Background}
\label{subsec:background}
The study of the geometric and analytic properties of metric spaces is
a topic which has grown into many different trends, and is probably as
old as the study of fractal subsets of $\real^n$ and of Carnot
groups. In the last 15 years, a trend which has attracted growing
interest is the study of metric measure spaces which admit an abstract
Poincar\'e inequality in the sense of \cite{heinonen98},
and which we will call PI-spaces.
\par Intuitively, PI-spaces allow to formulate notions of first-order
calculus, an intuition that was made more precise when
\cite{cheeger99} proved that PI-spaces satisfy a generalized
version of the classical Rademacher Theorem about the
a.e.~differentiability of real-valued Lipschitz functions. In particular,
Cheeger's result allows to associate to a PI-space $(X,\mu)$
$\mu$-measurable tangent / cotangent bundles $TX$ / $T^*X$, the
fibres of $T^*X$ being generated by ``differentials'' of Lipschitz
functions. It is worth to point out that Cheeger's generalization of
Rademacher's Theorem does not put $TX$ and $T^*X$ on an equal footing,
e.g.~derivatives are not explicitly constructed and are not related to
differentiation along Lipschitz curves. Essentially, differentiability
is formulated in terms of finite-dimensionality results for certain
spaces of (asymptotically) harmonic functions. Note that even though
this approach might look at first counterintuitive, it fits with the
idea that coordinate functions generate the cotangent bundle of a
Riemannian manifold as other functions admit a
first-order Taylor expansion with respect to the coordinates, and
moreover in Riemannian geometry there are
important results on finite-dimensionality of spaces of harmonic
functions whose proofs share some similarities with Cheeger's argument (see for instance
\cite{colding_harmonic}).
\par Today metric measure spaces which satisfy the conclusion
of Cheeger's Differentiation Theorem are either said to admit a
(strong) measurable differentiable structure \cite{keith_thesis,keith04},
or to be (Lipschitz) differentiability spaces \cite{bate-diff,bate_thesis_final};
in the following we will use the latter terminology.
\par In his PhD thesis, Keith \cite{keith_thesis,keith04} introduced a new
analytic condition, the Lip-lip inequality, and proved that doubling
metric measure spaces $(X,\mu)$ satisfying it are differentiability
spaces. It seems that the idea of ``generalizing'' Cheeger's
Differentiation Theorem using a Lip-lip inequality stems from the fact
that in PI-spaces Cheeger had proven a Lip-lip equality. We use
``generalize'' to refer to Keith's work because as of today there seem
to be no examples of differentiability spaces which cannot be
partitioned into countable unions of subsets, each of which admits a
measure-preserving biLipschitz embedding into some PI-space.
\par Besides providing a theoretical framework for first-order
calculus, the idea of differentiating Lipschitz functions has proven useful in the study of metric
embeddings $F:X\to B$ where $B$ is a Banach space, in particular when
either $B$ has the Radon-Nikodym property (i.e.~an RNP Banach space) (see for instance\cite{cheeger_kleiner_radon}), or
when $B=L^1$ (see for instance
\cite{cheeger_kleiner_L1_BV,cheeger_kleiner_metricdiff_monotone,cheeger_kleiner_naor,cheeger_inverse_l1}).
In connection with embeddings into
RNP-Banach spaces, Cheeger and Kleiner \cite{cheeger_kleiner_radon} showed that if $(X,\mu)$ is a
PI-space the fibres of $TX$ are spanned by ``tangent vectors'' to
Lipschitz curves. Putting $TX$ and $T^*X$ on a complete equal footing
has required substantial effort: Bate's beautiful work
\cite{bate-diff,bate_thesis_final} on Alberti representations
in differentiability spaces, which was partly motivated by a deep
structure theory for measures and sets in $\real^n$ developed by
Alberti, Cs{\"o}rnyei and Preiss
\cite{acp_plane,acp_proceedings}, and the formulation of metric differentiation
for differentiability spaces \cite{cks_metric_diff}, which was partly
motivated by unpublished results of Cheeger and Kleiner on metric differentiation
in PI-spaces, and unpublished results of mine on prescribing the norms on
$TX$ and $T^*X$.
\par Strikingly, Bate \cite{bate-diff,bate_thesis_final} was able to show that each differentiability
space $(X,\mu)$ can be partitioned into countably many pieces, each of
which is a doubling metric measure space admitting a Lip-lip
inequality. Later we showed \cite{schioppa_thesis,deralb} that the Lip-lip
inequality always self-improves to an equality (see also
\cite{cks_metric_diff} for another argument); this might be interpreted
as saying that the Lip-lip
equality provides an asymptotically quantitative characterization of
differentiability spaces; however, there is a more precise result in
terms of the quantitative characterization of the local norm for
Weaver forms (Theorem~\ref{thm:char_loc_norm}) which will be used in
this paper.
\par In connection with these results, a topic of major interest is
trying to understand the infinitesimal structure of differentiability,
and even PI-spaces. It should be pointed out that, as of today, the
set of known models for the infinitesimal geometry of PI-spaces is
rather limited, and is thus hard to come up with ``plausible''
conjectures. For example, while the results in \cite{cks_metric_diff}
show structural similarities between PI-spaces and differentiability
spaces (e.g.~one might conjecture that tangents/blow-ups of
differentiability spaces are PI spaces), the examples
constructed in \cite{schioppa_necks} suggest that there might
be differentiability spaces $(X,\mu)$ whose blow-ups are never
PI-spaces; moreover, this phenomenon might even depend on the measure
class of $\mu$. While finishing this paper, we have learned from Bate
and Li \cite{bate_rnp_new_definition} that they have studied the class
of differentiability spaces for which
differentiation of RNP-valued functions (in addition to that of
real-valued ones) holds, and have characterized
them in terms of an ``infinitesimal accessibility'' condition.
One possibility is that differentiability implies
RNP-differentiability, and then extending the results of
\cite{cheeger_kleiner_radon} to general differentiability
spaces might give a route to answer some questions raised in
\cite{cks_metric_diff}. Another possibility is that differentiability
spaces are organized in a sort of ``hierarchy'' depending on which
Banach-valued Lipschitz maps are differentiable. For instance, the
argument in \cite{bate_rnp_new_definition} uses crucially
differentiability in the $l_1$-sum of a sequence of finite-dimensional
spaces $\{l_\infty^{n_k}\}_k$ where $n_k\nearrow\infty$, and it does not seem clear
how one would recover the result in \cite{bate_rnp_new_definition} assuming, say,
differentiability of $l_2$-valued maps.
\subsection*{Results}
\label{subsec:results}
The main result of this paper states that if $(X,\mu)$ is a
differentiability space, then at $\mu$-a.e.~$x\in X$ any blow-up
$(Y_x,\nu_x)$ of $(X,\mu)$ is still a differentiability space. The
precise statement is Theorem~\ref{thm:blow_up_diff}.
\par The question of whether blow-ups of differentiability spaces are
still differentiability spaces has been around since Keith introduced
the Lip-lip inequality. This question should be compared with the
easier case where $(X,\mu)$ is a PI-space: then any blow-up is still
a PI-space. This follows from the stability of the Poincar\'e
inequality (with uniform constants) under measured Gromov-Hausdorff
convergence, which can be seen using Keith's elegant characterization
of PI-spaces in terms of moduli of families of curves \cite{keith-modulus}. Essentially,
the argument reduces to the upper-semicontinuity of modulus (which is
dual to the lower-semicontinuity of length). On the
other hand, in the category of differentiability spaces one cannot
expect all blow-ups to be differentiability spaces. For example, any
subset $S\subset\real^n$ of positive Lebesgue measure is a
differentiability space (restricting the Lebesgue measure), but
blow-ups are differentiability spaces (explicitly, copies of $\real^n$
with Lebesgue measure) only at generic points. Moreover, at the moment
there is no ``good geometric'' characterization of differentiability
spaces that one might pass to the limit; one has to work directly with
the definition of a differentiability space by showing that if at
$\mu$-a.e.~point $x$ there is some tangent $(Y_x,\nu_x)$ which is not
a differentiability space, then $(X,\mu)$ is not a differentiability
space. This can be done from two equivalent points of view:
\cite[Sec.~4]{bate-diff} choosing a candidate chart and producing a
non-differentiable function for that chart, \cite[Sec.~5.3]{deralb}
showing that the space of ``germs'' of Lipschitz functions is
infinite-dimensional.
\par The methodology that we propose to tackle this problem has two
components: ``lifting'' and ``gluing''. Lifting means to find a
suitable ``bad'' function $g_x$ on a blow-up $(Y_x,\nu_x)$, and then
lift it to a bad function $f_x$ on $(X,\mu)$. Unfortunately, $f_x$
will be bad only at some locations near $x$ and at some scales (which
can be small but are bounded away from $0$); thus it is then necessary to glue
several $f_x$'s together. It is worth to point out that lifting
depends on the structure theory for Weaver derivations developed in
\cite{schioppa_thesis,deralb}. In fact, $(Y_x,\nu_x)$ is not a
differentiability space, but there is still a decent theory (in
particular finite-dimensionality) for the $L^\infty$-modules of Weaver
derivations and forms. Another key ingredient in lifting is Preiss'
phenomenon (Theorem~\ref{thm:blow_ups_are_blow_ups}) that shifted
rescalings of blow-ups are still blow-ups.
\par Gluing consists in combining several $f_x$'s together. The
proposed approach is based on the idea of ``tile'' (see
Section~\ref{sec:loc_app}): for each $x$ one has to produce several
$f_x$'s at several scales converging to $0$, and then one applies
the Vitali Covering Lemma to join the pieces together. Intuitively, we
are using measure theory (e.g.~the measure-theoretic statement that
$(X,\mu)$ is finite-dimensional) to select the pieces that can be
glued together; in some sense, the construction can't be
``deterministic'' because Lipschitz functions are rather rigid and
$(X,\mu)$ is a generic doubling metric measure space. 
\par We mention some further directions of research. We  are able to
show that for iterated blow-ups of differentiability spaces the
analytic dimension does not increase (see Lemma~\ref{lem:regular_diff}). It seems plausible to
conjecture that one does not need to take iterated blow-ups: at the
moment we  are able to prove this only under additional assumptions on
$(X,\mu)$ (e.g.~when $(X,\mu)$ biLipschitz embedds in Carnot groups or
``nice'' Banach spaces), and thus do not include these partial results
here. Essentially, there seems to be a technical
obstacle in directly applying the ``lifting'' method in the form
proposed here. This question is also related to wheteher the results
in \cite{cheeger_kleiner_radon} extend to general differentiability spaces,
and on Lusin-like properties for Weaver derivations that will be
discussed elsewhere.
\par As an application, we obtain a characterization of the $p$-weak
gradient of a Lipschitz function $f$ on regular differentiability
spaces (Definition~\ref{defn:regular_diff}) which arise as iterated
blow-ups of differentiability spaces.
\par This is related to another trend in analysis on metric spaces,
where people have tried to formulate generalizations of Sobolev/BV
spaces and gradients. One possible approach starts with the idea
\cite{heinonen98} of an upper gradient for a function $f$, which is
essentially an upper bound on the norm of ``what the gradient of $f$
should be''. Then there have been several proposals for ``what the norm
of the gradient should be'': a variational one due to
\cite{cheeger99}, one motivated by quasiconformal geometry and
moduli of curves due to Koskela, MacManus and Shanmugalingam \cite{kosk_mac_quasiconf_sob,shanmugalingam_newton}, and
two motivated by optimal transport due to
Ambrosio, Gigli and Savar\'e \cite{ambrosio_weakgrads}. Optimal transport allows to show
\cite{ambrosio_weakgrads} that all these approaches are equal under mild
assumptions on $(X,\mu)$, e.g.~assuming $X$ to be complete and
separable, and $\mu$ to be finite on bounded sets. In this paper we
work with the $p$-weak gradient of \cite{ambrosio_weakgrads} because it is
easier to relate to the existence of Alberti representations using the
notion of test plan.
\par Asking for a characterization of the $p$-weak gradient on
differentiability spaces is not a good question because positive
measure subsets of differentiability spaces are still
differentiability spaces, making the notion of $p$-weak gradient often
vacuous (e.g.~considering a Vitali-Cantor set). However, an
interesting question is to ask for a characterization of the $p$-weak
gradient on blow-ups of differentiability spaces. We show
(Theorem~\ref{thm:weak_grad_indep}) that for regular differentiability
spaces the $p$-weak gradient of a Lipschitz function $f$ coincides
with the asymptotic Lipschitz constant, and is hence independent of
the exponent $p$: this generalizes a previous result
\cite{cheeger99} for PI-spaces. Note however, that for
PI-spaces the $p$-weak gradient coincides with the asymptotic
Lipschitz constant only if $p$ belongs to the range of exponents for
which the Poincar\'e inequality holds: this has been shown by recent
examples \cite{mar_spe_grad_depends}. This is not the case for regular
differentiability spaces: the $p$-weak gradient does not depend on
$p$. One can explain the examples \cite{mar_spe_grad_depends} in terms of
putting ``bad weights'' on the lines corresponding to Alberti
representations, but these bad weights become again ``nice'' by
passing to tangent measures. We point out that
Theorem~\ref{thm:weak_grad_indep} has an analogue in the BV category,
see Remark~\ref{rem:bv} (see \cite{ambrosio_bvgrads} for extensions of
\cite{ambrosio_weakgrads} to the BV category).
\par As a historical note, we point out that the notion of regular
differentiability space refines and generalizes the notion of
generalized Minkowski space in \cite[Sec.~11]{cheeger99}.
\subsection*{Organization}
\label{subsec:organization}
In Section~\ref{sec:background} we discuss background material on
Weaver derivations, Alberti representations, my PhD thesis,
differentiability spaces and the $p$-weak gradient. The presentation
is a bit brisk, so we invite the interested reader to consult the
references therein.
\par In Section~\ref{sec:blow_weaver} we first
discuss (subsection~\ref{subsec:mGHconvs}) variants of Gromov-Hausdorff
convergence, the purpose being mainly to establish some
terminology. The substantial result from this subsection that we will
use is the aforementioned Preiss'
phenomenon, Theorem~\ref{thm:blow_ups_are_blow_ups}
\cite{preiss_geometry_measures,ledonne_tangents,gigli_weak_euclidean_tangs}.
We then move on
(subsection~\ref{subsec:bwp_wea_norm}) to explain how rescalings
affect the modules of derivations and forms. We conclude this section
with a generalization of a result \cite[Sec.~7]{cks_metric_diff} to the category
of Weaver derivations. We point out that this result to blow-up Weaver
derivations has other applications, e.g.~to the infinitesimal
structure of metric currents in $\real^n$, that will be discussed
elsewhere.
\par Section~\ref{sec:loc_app} contains our proposal to implement the
``gluing'' part of the argument; this is more general than what we
really use here, because we end up working with ``cubical'' tiles
\cite{christ_cubes,hytonen_cubes}. Using other geometries for tiles
might lead to a better understanding of the geometry of blow-ups; for example, Preiss and I discussed sometime
ago what are essentially ``long cylindrical'' tiles to exclude
factorizations of the form $Y\times\real^n$ in blow-ups of
differentiability spaces.
\par In subsection~\ref{subsec:diff_blow_diff} it is shown that generic
blow-ups of differentiability spaces are still differentiability
spaces by proving the contrapositive statement; the first $6$ steps of
Theorem~\ref{thm:blow_up_diff} correspond to
``lifting''. Subsection~\ref{subsec:weak_grad_indep} discusses  the
characterization of the $p$-weak gradient on regular differentiability
spaces; the key step is Lemma~\ref{lem:test_plan_arep} which
associates test plans to ``nice'' Alberti representations.
\subsection*{Notational conventions}
\label{subsec:notation}
We use the convention $a\approx b$ to say that $a/b,b/a\in[C^{-1},C]$
where $C$ is a universal constant; when we want to highlight $C$ we
write $a\approx_Cb$. We similarly use notations like $a\lesssim b$ and
$a\gtrsim_Cb$.
\subsection*{Acknowledgements}
\label{subsec:ackwn}
This work would have not been possible without the many conversations
that I had with David Preiss, who generously invited me to visit the
University of Warwick; I also wish to thank the people in the analysis group,
in particular Daniel Seco, for the hospitality I received.
\par I also thank the anonymous referee for reading the manuscript
very carefully and for pointing out an issue with the way measures
were normalized in the first version of the preprint.
\section{Background Material}
\label{sec:background}
\subsection{Weaver derivations}
\label{subsec:derivations}
\par For more information we refer the reader to \cite{weaver00}.
An \textbf{$L^\infty(\mu)$-module} $M$ is a Banach space $M$ which
is also an $L^\infty(\mu)$-module and such that for all
$(m,\lambda)\in M\times L^\infty(\mu)$ one has:
\begin{equation}
  \label{eq:boundedaction}
  \|\lambda m\|_M\le\|\lambda\|_{L^\infty(\mu)}\,\|m\|_M.
\end{equation}
Among $L^\infty(\mu)$-modules a special r\^ole is played by
\textbf{$L^\infty(\mu)$-normed modules}:
\begin{defn}[Normed modules]
  \label{defn:local_norm}
  An $L^\infty(\mu)$-module $M$ is said to be an
  \textbf{$L^\infty(\mu)$-normed module} if there is a map
  \begin{equation}
    |\cdot|_{M,{\text{loc}}}:M\to L^\infty(\mu) 
  \end{equation}
such that:
\begin{enumerate}
\item For each $m\in M$ one has $|m|_{M,{\text{loc}}}\ge0$;
\item For all
  $c_1,c_2\in\real$ and $m_1,m_2\in M$ one has:
  \begin{equation}
        |c_1m_1+c_2m_2|_{M,{\text{loc}}}\le|c_1||m_1|_{M,{\text{loc}}}+|c_2||m_2|_{M,{\text{loc}}};
  \end{equation}
\item For each $\lambda\in L^\infty(\mu)$ and each $m\in M$, one has:
  \begin{equation}
    |\lambda m|_{M,{\text{loc}}}=|\lambda|\,|m|_{M,{\text{loc}}};
  \end{equation}
\item The local seminorm $|\cdot|_{M,{\text{loc}}}$ can be used to
  reconstruct the norm of any $m\in M$:
  \begin{equation}
    \|m\|_M=\|\,|m|_{M,{\text{loc}}}\,\|_{L^\infty(\mu)}.
  \end{equation}
\end{enumerate}
\end{defn}
\begin{defn}[Weaver derivation]
  \label{defn:derivations}
    A \textbf{derivation $D:\lipalg X.\to L^\infty(\mu)$} is a weak*
    continuous, bounded linear map satisfying the product rule:
    \begin{equation}
      D(fg)=fDg+gDf.
    \end{equation}
  \end{defn}
  \par Note that the product rule implies that $Df=0$ if $f$ is
  constant. The collection of all derivations $\wder\mu.$ is an
  $L^\infty(\mu)$-normed module
  \cite[Thm.~2]{weaver00}
  and the corresponding  local norm will be denoted by
$\locnorm\,\cdot\,,{\wder\mu.}.$. Note also that $\wder\mu.$ depends
only on the measure class of $\mu$.
\begin{rem}
  \label{rem:restr}
  Consider a Borel set $U\subset X$ and a derivation
  $D\in\wder\mu\on U.$. The derivation $D$ can be also regarded as
  an element of $\wder\mu.$ by extending $Df$ to be $0$ on $X\setminus
  U$. In particular,
  the module $\wder\mu\on U.$ can be naturally identified with the
  submodule $\chi_U\wder\mu.$ of $\wder\mu.$.
\end{rem}
\par Derivations are local in the following sense
\cite[Lem.~27]{weaver00}):
\begin{lem}[Locality of Derivations]
  \label{lem:locality_derivations}
  If $U$ is $\mu$-measurable and if $f,g\in\lipalg X.$ agree on $U$,
  then for each $D\in\wder\mu.$, $\chi_UDf=\chi_UDg$.
\end{lem}
Note that locality allows to extend the action of derivations on
Lipschitz functions so that if $f\in\lipfun X.$ and $D\in\wder\mu.$,
$Df$ is well-defined (see
Remark \sync rem:derivation_extension.\  in
                                \cite{deralb}).
We now
pass to consider some algebraic properties of $\wder\mu.$.
\par In general, even if the module $\wder\mu.$ is finitely generated,
it is not free. Nevertheless, it is possible to obtain a decomposition
into free modules over \emph{smaller rings} \cite{weaver00,derivdiff}
:
\begin{thm}[Free Decomposition]
  \label{thm:free_dec}
  Suppose that the module $\wder\mu.$ is finitely generated with $N$ generators. Then
  there is a Borel partition $X=\bigcup_{i=0}^N X_i$ such that, if
  $\mu(X_i)>0$, then $\wder\mu\on X_i.$ is free of rank $i$ as an
  $L^\infty(\mu\on X_i)$-module. A
  basis of $\wder\mu\on X_i.$ will be called a {\normalfont\textbf{local basis of
    derivations}}.
\end{thm}
\begin{rem}
  \label{rem:index}
  In particular, Theorem~\ref{thm:free_dec} can be applied if one knows
  an upper bound on the \textbf{index} of $\wder\mu.$ which is defined
  as follow:
  \begin{multline}
    \label{eq:index_1}    
    \text{index of $\wder\mu.$} = \sup\{
    n\in\natural: \text{$\exists U$ Borel: $\wder\mu\on U.$ contains
      $n$-independent}\\ \text{elements (over $L^\infty(\mu\on U)$)
      }
    \}.
  \end{multline}
  \par In many applications in analysis on metric spaces the assumption
that $\wder\mu.$ has finite index (and is hence finitely generated) is
not restrictive: for example
it holds if either $\mu$ or $X$ are doubling  (see Corollary \sync cor:derbound.\ in
 \cite{deralb}).
\end{rem}
\par We now recall the notion of $1$-forms which are dual to
derivations.
\begin{defn}
  \label{defn:module_forms}
  The \textbf{module of $1$-forms} $\wform\mu.$ is the dual module of
  $\wder\mu.$, i.e.~it consists of the bounded module homomorphisms
  $\wder\mu.\to L^\infty(\mu)$. The module $\wform\mu.$ is an
  $L^\infty(\mu)$-normed module and the local norm will be denoted by
  $\locnorm\,\cdot\,,{\wform\mu.}.$.
  \par To each $f\in\lipalg X.$ one can associate the $1$-form
  $df\in\wform\mu.$ by letting:
  \begin{equation}
    \label{eq:module_forms1}
    \langle df,D\rangle=Df\quad(\forall D\in\wder\mu.);
  \end{equation}
  the map $d:\lipalg X.\to\wform\mu.$ is a weak* continuous $1$-Lipschitz
  linear map satisfying the product rule $d(fg)=gdf+fdg$.
\end{defn}
Note that because of Lemma \ref{lem:locality_derivations} one can
extend the domain of $d$ to $\lipfun X.$ so that if $f$ is Lipschitz,
$df$ is a well-defined element of $\wform\mu.$ and
$\|df\|_{\wform\mu.}\le\glip f.$, 
{where $\glip f.$ denotes the global
Lipschitz constant of $f$.}
\subsection{Alberti representations}
\label{subsec:alberti_reps}
Alberti representations (without this name) were introduced in
\cite{alberti_rank_one}; we invite the reader to consult
\cite{bate_thesis_final,deralb,cks_metric_diff} for more information.
\begin{defn}[Fragments and Curves]
  \label{defn:frags}
A {\bf fragment in $X$}
is a Lipschitz map $\gamma:C\to X$, where $C\subset\real$ is closed. The
set of fragments in $X$ will be denoted by $\frags(X)$. We now discuss
the topology on $\frags(X)$;  let $F(\real\times X)$ denote the set of
closed subsets of $\real\times X$ with the Fell topology
\cite[(12.7)]{kechris_desc}
; we recall that a basis of the Fell
topology consists those sets of the form:
\begin{equation}
  \label{eq:felltop}
  \left\{F\in F(\real\times X): F\cap K=\emptyset, F\cap
    U_i\ne\emptyset\;\text{for $i=1,\ldots,n$}\right\},
\end{equation}
where $K$ is a compact subset of $\real\times X$, and 
$\{U_i\}_{i=1}^n$ is a finite collection of open subsets of
$\real\times X$. Each fragment $\gamma$ can be identified with an element
of $F(\real\times X)$ and thus $\frags(X)$ will be topologized as
a subset of $F(\real\times X)$. We will use fragments to
parametrize $1$-rectifiable subsets of $X$.
\par An important subset of
$\frags(X)$ consists of the Lipschitz curves $\curves(X)$, which is the set of those $\gamma\in\frags(X)$
whose domain is a (possibly unbounded) closed subinterval of
$\real$. Given an interval $I\subset[0,1]$ we denote by $\curves(X,I)\subset\curves(X)$
the set of those $\gamma\in\curves(X)$ whose domain is contained in $I$.
\end{defn}
\begin{defn}[Alberti representations]
  \label{defn:alb_rep}
  Let $\mu$ be a Radon measure on $X$. An \textbf{Alberti representation} of
  $\mu$ is a pair $\albrep.=[Q,w]$ where $Q$ is a Radon measure on
  $\frags(X)$ and $w$ a Borel function $w:X\to[0,\infty)$ such that:
  \begin{equation}
    \label{eq:alb_rep_1}
    \mu=\int_{\frags(X)}w\cdot\gamma_{\#}(\lebmeas.\on\dom\gamma)\,dQ(\gamma),
  \end{equation}
  where the integral is interpreted in the weak* sense. We say that
  $\albrep.$ is \textbf{$C$-Lipschitz} (resp.~\textbf{$[C,D]$-biLipschitz}) if $Q$ is concentrated on the
  set of $C$-Lipschitz (resp.~$[C,D]$-biLipschitz) fragments. 
\end{defn}
\begin{defn}[Speed of fragments]
  \label{defn:met_diff}
  Let $\gamma\in\frags(X)$; then for $\lebmeas.$-a.e.~$t\in\dom\gamma$
  the limit:
  \begin{equation}
    \label{eq:met_diff_1}
    \lim_{\dom\gamma\ni t'\to t}\frac{
      d(\gamma(t'),\gamma(t))
    }
    {
      |t'-t|
    } 
  \end{equation}
  exists, is denoted by $\metdiff\gamma(t)$ and is called the
  \textbf{metric differential} of $\gamma$ at $t$.
\end{defn}
\begin{defn}[Speed of Alberti representations]
  \label{defn:alb_speed}
  Let $\albrep.=[Q,w]$ be an Alberti representation, and let
  $\sigma:X\to[0,\infty)$ be a Borel function and $f:X\to\real$ be a
  Lipschitz function. We say that $\albrep.$ has $f$-speed $\ge\sigma$
  (resp.~$\le\sigma$) if for $Q$-a.e.~$\gamma$ one has that for
  $\lebmeas.$-a.e.~$t\in\dom\gamma$:
  \begin{equation}
    \label{eq:alb_speed_1}
    (f\circ\gamma)'(t)\ge\sigma(\gamma(t))\metdiff\gamma(t)\quad(\text{resp.~$\le\sigma(\gamma(t))\metdiff\gamma(t)$}).
  \end{equation}
\end{defn}
\subsection{The correspondence between derivations and Alberti
  representations}
\label{subsec:correspondence}
We invite the reader to consult \cite{schioppa_thesis,deralb} for more
information. The following definition follows from Theorem
\sync thm:alb_derivation.\ in \cite{deralb}.
\begin{defn}[Derivation associated to an Alberti representation]
  \label{defn:alberti_to_derivation}
  Let $\albrep.=[Q,w]$ be a $C$-Lipschitz Alberti representation of a measure
  $\nu\ll\mu$; then the formula:
  \begin{multline}
    \label{eq:alberti_to_derivation_1}
    \int_X gD_{\albrep.}f\,d\nu =
    \int_{\frags(X)}dQ(\gamma)\int_{\dom\gamma}(wg)\circ\gamma(t)
    (f\circ\gamma)'(t)\,dt\\
    (\forall (g,f)\in C_c(X)\times\lipalg X.)
  \end{multline}
  defines a derivation $D_{\albrep.}\in\wder\nu.\subset\wder\mu.$ with
  $\|D_{\albrep.}\|_{\wder\mu.}\le C$.
\end{defn}
\pcreatenrm{wfx}{\wform\mu;X.}{|}
\pcreatenrm{wf}{\wform\mu.}{|}
The following theorem summarizes some results in
\cite{schioppa_thesis,deralb} that we will use.
\begin{thm}[Correspondence between derivations and representations]
  \label{thm:der_alb_corr}
  Let $\mu$ be a Radon measure on a complete separable metric measure
  space $(X,\mu)$. Then:
  \begin{description}
  \item[(IndBound)] If $\mu$ or $X$ are doubling (with constant $C$),
    then the index of $\wder\mu.$ is $\le\lfloor\log_2C\rfloor$; 
  \item[(Surj)] If $\wder\mu.$ is finitely generated with $N$
    generators, then for each $\varepsilon>0$ and each $D\in\wder\mu.$
    there is an
    $(1+\varepsilon)N\|D\|_{\wder\mu.}$-Lipschitz Alberti
    representation $\albrep.$ of a measure $\nu\ll\mu$ such that
    $D_{\albrep.}=D$;
  \item[(WDens)] In general the set of derivations of the form
    $D_{\albrep.}$ is weak* dense in $\wder\mu.$ in the sense that,
    given any $D\in \wder\mu.$, 
    {for all $\varepsilon>0$}, for each $\lambda\in L^1(\mu)$ and for
    each finite set of Lipschitz functions $\{g_i\}_{i=1}^k$ one can
    find an Alberti representation $\albrep.$ of $\mu$ such that:
    \begin{equation}
      \label{eq:der_alb_corr_s1}
      \int_X|\lambda|\sum_{i=1}^k|Dg_i-D_{\albrep.}g_i|\,d\mu\le\varepsilon.
    \end{equation}
  \end{description}
\end{thm}
Recall that a sequence $\{f_n\}\subset\lipalg X.$ converges to
$f\in\lipalg X.$ in
the weak* topology ($f_n\wkcvj f$) if $f_n\to f$ pointwise and
$\sup_n\glip f_n.<\infty$.
This result is a functional-analytic interpretation of constructions
in \cite{acp_plane,acp_proceedings,bate-diff}, that was obtained in
\cite{schioppa_thesis,deralb}; to obtain the optimal constants we
were greatly helped by Andrea Marchese's PhD thesis.
\begin{thm}[Approximation Scheme]
  \label{thm:approx}
  {Let $K\subset X$ be compact and assume that $\mu\on K$ does not admit
  an Alberti representation with $f$-speed $\ge\delta$. Then there is
  a sequence $g_k\wkcvj f$ such that each $g_k$ is $\max(\glip
  f.,\delta)$-Lipschitz and, for each $k$, there is an open
  set $U_k\supset K$ such that for each ball $B\subset U_k$ the
  restriction 
  $g_k|B$ is $\delta$-Lipschitz.}
\end{thm}
This result  \cite[Sec.~3.3]{schioppa_thesis,deralb} characterizes the
local norm on the Weaver cotangent bundle.
\begin{thm}[Characterization of the local norm]
  \label{thm:char_loc_norm}
  Let $K\subset X$ be compact, $f:X\to\real$ Lipschitz and
  $\varepsilon>0$. Then $\mu\on K$ admits a
  $(1,1+\varepsilon)$-biLipschitz Alberti representation with
  $f$-speed $\ge\wfnrm df.-\varepsilon$.
\end{thm}
\subsection{Differentiability spaces}
\label{subsec:diff_spaces}
We start with a brief review of differentiability spaces. 
For more details we refer to the original
papers \cite{cheeger99,keith_thesis,keith04} or to the expository paper
\cite{kleiner_mackay}. This structure has several names in
the literature: (strong) measurable differentiabile structure,
differentiable structure (in the sense of Cheeger and Keith),
Lipschitz differentiability space, differentiability space.
We highlight the features of differentiability spaces; contrary to
some earlier papers, we do not assume a uniform bound on the dimension of the
charts.
\begin{enumerate}
\item There is a countable collection of charts
  $\{(U_\alpha,\phi_\alpha)\}_\alpha$, 
  where $U_\alpha\subset X$ is Borel and $\phi_\alpha$ is Lipschitz,
  such that $X\setminus
  (\cup_\alpha U_\alpha)$ is $\mu$-null, and each real-valued
  Lipschitz function $f$ admits a first order Taylor expansion with
  respect to the components of $\phi_\alpha:X\to\real^{N_\alpha}$
  at
  generic points of $U_\alpha$, i.e.~there are (a.e.~unique) measurable functions $\frac{\partial
      f}{\partial\phi_\alpha^i}$ on $U_\alpha$ such that:
  \begin{multline}
    \label{eq:taylor_exp}
    f(x)=f(x_0)+\sum_{i=1}^{N_\alpha}\frac{\partial
      f}{\partial\phi_\alpha^i}(x_0)\left(\phi_\alpha^i(x)-\phi_\alpha^i(x_0)\right)
    +o\left(d(x,x_0)\right)\\ \quad(\text{for $\mu$-a.e.~$x_0\in U_\alpha$}).
  \end{multline}
  The integer $N_\alpha$ is the dimension of the chart
  $\{(U_\alpha,\phi_\alpha)\}_\alpha$, and depends only on the set
  $U_\alpha$, not on the particular choice of the coordinate functions
  $\phi_\alpha$. If $\sup_\alpha N_\alpha<\infty$, it is called \textbf{the
  differentiability or the analytic dimension}. Since in this paper we
  are interested in the blow-ups of differentiability
  spaces, we can assume that there is a unique chart that covers all
  the space.
\item There are  measurable cotangent and tangent bundles
  $T^*X$ and $TX$; however, we will work with $\wder\mu.$ and
  $\wform\mu.$. By the result of \cite{deralb,schioppa_thesis} $\wder\mu.$
  can be identified with the set of bounded measurable sections of
  $TX$ and $\wform\mu.$ with the set of bounded measurable sections of
  $T^*X$. Having locally trivialized $T^*X$ and $TX$, forms in $T^*X$
  correspond to differentials of Lipschitz functions, and vectors in
  $TX$ give rise to examples of Weaver derivations in $\wder\mu.$.
\end{enumerate}
\def\varlip{{\rm Var}} 
\begin{defn}[Variation and pointwise Lipschitz constants]
  \label{defn:var_lipcons}
  For $f\in\lipfun X.$ we define the \textbf{variation $\varlip f(x,r)$ at $x$ at scale $r$},
the \textbf{big-Lip constant $\biglip f$ at $x$} and the
\textbf{small-lip constant $\smllip f$ at $x$} as follows: 
\begin{align}
  \label{eq:var_lipcons_1}
  \varlip f(x,r) &= \frac{1}{r} \sup_{y\in B(x,r)}|f(y)-f(x)|\\
  \biglip f(x)&=\limsup_{r\searrow0}\varlip f(x,r)\\
  \smllip f(x)&=\liminf_{r\searrow0}\varlip f(x,r).
\end{align}
A metric measure space $(X,\mu)$ satisfies Keith's \textbf{Lip-lip
  inequality} with constant $K\ge1$ if for each $f$ Lipschitz one has:
\begin{equation}
  \label{eq:var_lipcons_2}
  \biglip f\le K\smllip f\quad\text{$\mu$-a.e.}
\end{equation}
\end{defn}
\begin{thm}[Summary of results on differentiability spaces]
  \label{thm:diff_summary}
  This list summarizes relevant results on differentiability spaces:
  \begin{description}
  \item[(Cheeger)] \cite{cheeger99}; if $(X,\mu)$ is a doubling metric
    measure space which admits an abstract Poincar\'e
    inequality in the sense of Heinonen-Koskela \cite{heinonen98}
    then $(X,\mu)$ is a differentiability space whose analytic
    dimension is bounded by an expression that depends only on the
    doubling constant $C_\mu$ of $\mu$ and the constants that appear
    in the Poincar\'e inequality. Moreover, (\ref{eq:var_lipcons_2})
    holds with $K=1$;
  \item[(Keith)] \cite{keith_thesis,keith04}; if $(X,\mu)$ is a
    doubling metric measure space which
    satisfies the Lip-lip inequality~(\ref{eq:var_lipcons_2}), then
    $(X,\mu)$ is a differentiability space whose analytic dimension is
    bounded by an expression that depends only on $C_\mu$ and
    $K$. Moreover \cite{keith-modulus}, the Poincar\'e inequality is
    stable under measured Gromov-Hausdorff convergence provided all
    the relevant constants are uniformly bounded; for example,
    blow-ups of PI-spaces are PI-spaces (with the same PI-exponent);
  \item[(Bate--Speight)] \cite{bate_speight}; if $(X,\mu)$ is a
    differentiability space then $\mu$ is asymptotically doubling in
    the sense that for $\mu$-a.e.~$x$ there are
    $(C_x,r_x)\in(0,\infty)^2$ such that:
    \begin{equation}
      \label{eq:diff_summary_s1}
        \mu\left(
          B(x,2r)
        \right) \le C_x
        \mu\left(
          B(x,r)
        \right)
        \quad(r\le r_x);
      \end{equation}
    moreover, porous sets are $\mu$-null. 
    In the following, to simplify the exposition, we will thus assume
    that differentiability spaces are doubling metric measure spaces.
  \item[(Bate)] \cite{bate-diff,bate_thesis_final}; if $(X,\mu)$ is a differentiability space, $\mu$
    admits many independent Alberti representations generalizing some
    of the
    results in $\real^N$ of \cite{acp_plane,acp_proceedings}. Moreover,
    (\ref{eq:var_lipcons_2}) holds where now $K$ is a Borel function
    that depends only on $X$;
  \item[(Schioppa)] \cite{schioppa_thesis,deralb}; in a differentiability
    space~(\ref{eq:var_lipcons_2}) always holds with $K=1$. Moreover,
    $(X,\mu)$ is a differentiability space if and only if for each
    Lipschitz function $f$:
    \begin{equation}
      \label{eq:diff_summary_s2}
      \biglip f=\wfnrm df.\quad(\text{$\mu$-a.e.}),
    \end{equation}
    and~(\ref{eq:diff_summary_s2}) already encodes the condition that
    $\mu$ is asymptotically doubling. We will refer
    to~(\ref{eq:diff_summary_s2}) as the \textbf{quantitative
      characterization} of differentiability spaces.
  \end{description}
\end{thm}
As a historical note, we point out that an earlier result (where one
loses the optimal PI-exponent) on the
stability of the Poincar\'e inequality under measured Gromov-Hausdorff
convergence can be found in
\cite[Sec.9]{cheeger99}. 
\subsection{The $p$-weak gradient}
\label{subsec:weak_grad}
We recommend as a reference \cite{acm_sobreflex} besides \cite{ambrosio_weakgrads}.
\def\evala{\text{\normalfont Ev}} 
\begin{defn}[Absolutely continuous curves]
  \label{defn:ac_curves}
  {Let $I$ be either a nondegenerate interval of $\real$ of the form
  $[a,b]$, $(-\infty,a]$ or $[a,\infty)$, or the whole $\real$. A
  curve $\gamma:I\to X$ is \textbf{absolutely continuous} if there is
  a $g\in L^1(\lebmeas.\on\dom\gamma)$ such that for each
  $t,s\in\dom\gamma$ with $t\ge s$ one has:
  \begin{equation}
    \label{eq:ac_curves0}
    d(\gamma(t),\gamma(s))\le\int_s^t g(\tau)\,d\tau.
  \end{equation}
  Recall that if $\gamma:\dom\gamma\to X$ is absolutely continuous
  there is a minimal $g$ satisfying~(\ref{eq:ac_curves0}) which
  coincides  $\lebmeas.\on\dom\gamma$-a.e.~with the metric
  differential $\metdiff\gamma$ as defined in~(\ref{eq:met_diff_1}).
  Let $p\in[1,\infty)$ and $\gamma$ be an absolutely continuous
  curve; then $\gamma$ is of \textbf{class $\acspace^p$} if:
  \begin{equation}
    \label{eq:ac_curves_1}
    \int_{\dom\gamma}(\metdiff\gamma(t))^p\,dt<\infty.
  \end{equation}}
  The limit case $p=\infty$ corresponds to $\gamma$ being
  Lipschitz. The set of curves of class $\acspace^p$ and with domain
  $[0,1]$ will be denoted by $\acspace^p(X;[0,1])$.
\end{defn}
\begin{defn}[Test plan]
  \label{defn:test_plan}
  Let $(X,\mu)$ be a metric measure space and $p\in[1,\infty]$; a
  probability measure on $\acspace^p(X;[0,1])$ is called a
  \textbf{$p$-test plan} provided that:
  \begin{align}
    \label{eq:test_plan_1}
    \int_{\acspace^p(X;[0,1])}d\pi(\gamma) \int_0^1
    (\metdiff\gamma(t))^p\,dt&<\infty\\
    \label{test_plan_2}
    \evala_{t\#}\pi &\le C(\pi)\mu\quad(\forall t\in[0,1])
  \end{align}
  for some constant $C(\pi)$, where $\evala_t$ denotes the \textbf{evaluation
  map}:
  \begin{equation}
    \label{eq:test_plan_3}
    \begin{aligned}
      \evala_t:\acspace^p(X;[0,1])&\to X\\
      \gamma&\mapsto\gamma(t).
    \end{aligned}
  \end{equation}
\end{defn}
\pcreateasymnrm{pwgrad}{\text{\normalfont$p$,w}}{|\nabla}{|}
\begin{defn}[$p$-weak gradients]
  \label{defn:weak_grad}
  Let $(X,\mu)$ be a metric measure space, $f:X\to\real$ and
  $g:X\to[0,\infty]$ Borel. Let $p\in(1,\infty)$ and $q$ denote the
  dual exponent $\frac{p}{p-1}$. The function $g$ is a \textbf{$p$-weak
    upper gradient} of $f$ if for each $q$-test plan $\pi$
  \begin{equation}
    \label{eq:weak_grad_1}
    \left|
      f(\gamma(1)) -
      f(\gamma(0))
    \right| \le \int_0^1 g\circ\gamma\,\metdiff\gamma(t)\,dt
  \end{equation}
  holds $\pi$-a.s.
  \par Assuming that $f$ has a $p$-weak upper gradient $g_0$ such that the
  measure $\mu\on\{g_0>0\}$ is $\sigma$-finite, then the set of
  $p$-weak upper gradients of $f$ contains a minimal element, the
  \textbf{$p$-weak gradient of $f$}, such that:
  \begin{equation}
    \label{eq:weak_grad_2}
    \pwgradnrm f.\le g\quad(\text{$\mu$-a.e.})
  \end{equation}
  for each $p$-weak upper gradient $g$ of $f$.
\end{defn}
\section{Blow-up of Weaver Derivations}
\label{sec:blow_weaver}
\subsection{Variants of Gromov-Hausdorff convergence}
\label{subsec:mGHconvs}
\pcreatedst{y}{d_Y}
\pcreatedst{z}{d_Z}
\pcreatedst{x}{d_X}
\pcreatedst{xn}{d_{X_n}}
\pcreateconv{gh}{\text{\normalfont GH}}
\pcreateconv{mgh}{\text{\normalfont mGH}}
\pcreateconv{mfgh}{\text{\normalfont mfGH}}
We discuss some variants of Gromov-Hausdorff convergence (GH for
short). Throughout this subsection metric spaces are assumed to be
complete.
\begin{defn}[GH-convergence]
  \label{defn:gh_conv}
  The GH-convergence of a sequence of pointed metric spaces
  $(X_n,p_n)\ghcvj (Y,q)$ is equivalent to the existence of a pointed
  metric space $(Z,z)$, which we will call a \textbf{container}, and
  isometric embeddings:
  \begin{equation}
    \label{eq:gh_conv_1}
    \begin{aligned}
      \iota_n&:X_n\to Z\\
      \iota_\infty&:Y\to Z
    \end{aligned}
  \end{equation}
  such that:
  \begin{description}
  \item[(GH1)] $\iota_n(p_n)\to\iota_\infty(q)$; one can even arrange
    $\iota_n(p_n)=\iota_\infty(q)=z$;
  \item[(GH2)] for each $R>0$, one has:
    \begin{equation}
      \label{eq:gh_conv_3}
      \begin{aligned}
        &\lim_{n\to\infty}\sup_{y\in\ball z,R.\cap
          \iota(Y)}\dist\left(\iota_n\left(X_n\right),\{y\}\right)=0,\\
        &\lim_{n\to\infty}\sup_{y\in\ball z,R.\cap
          \iota_n(X_n)}\dist\left(\iota\left(Y\right),\{y\}\right)=0.
      \end{aligned}
    \end{equation}
  \end{description}
  In the following we will often suppress $\iota_n$ and $\iota_\infty$
  from the notation, just implying $X_n,Y\subset Z$. Note that 
  each $y\in Y$ can be ``approximated'' by a sequence
  $x_n\in X_n$ such that $\iota_n(x_n)\to\iota_\infty(y)$ in
  $Z$. This notion is actually independent of the container $(Z,z)$, and
  one can \textbf{represent} each point $y\in Y$ by some sequence
  $x_n\in X_n$.
\end{defn}
\begin{defn}[mGH-convergence]
  \label{defn:mgh_conv}
  {We say that a pointed metric measure space $(X,\mu,p)$ is
  \textbf{measure-normalized} if $\int_{B(p,1)}(1-d(p,x))\,d\mu(x)=1$.} The measured
  GH-convergence (mGH for short) of a sequence of measure-normalized pointed
  metric measure spaces $(X_n,\mu_n,p_n)\mghcvj (Y,\nu,q)$ is
  equivalent to requiring $(X_n,p_n)\ghcvj (Y,q)$, and that for each
  container $(Z,z)$:
  \begin{equation}
    \label{eq:mgh_conv_1}
    \mu_n\wkcvj \nu\quad(\text{i.e.~$\iota_{n,\#}\mu_n\wkcvj\iota_{\infty,\#}\nu$}).
  \end{equation}
\end{defn}
\begin{defn}[mfGH-convergence]
  \label{defn:mfgh_conv}
  A \textbf{function space} is a tuple $(X,\mu,p,\Phi)$ where
  $(X,\mu,p)$ is a normalized pointed metric measure space, and $\Phi$
  is an at most countable collection of real-valued $1$-Lipschitz
  functions on $X$ which vanish at $p$. We say that:
  \begin{equation}
    \label{eq:mfgh_conv_1}
    (X_n,\mu_n,p_n,\Phi_n)\mfghcvj (Y,\nu,q,\Psi)
  \end{equation}
  in the mfGH-sense if:
  \begin{description}
  \item[(mfGH1)] $(X_n,\mu_n,p_n)\mghcvj (Y,\nu,q)$;\\
  \item[(mfGH2)] $\Phi_n$ and $\Psi$ have eventually the same
    cardinality;
  \item[(mfGH3)] Let $\varphi_{n,k}$ denote the $k$-th element of
    $\Phi_n$ and $\psi_k$ the $k$-th element of $\Psi$; then
    whenever $x_n$ represents $y$, $\varphi_{n,k}(x_n)\to\psi_k(y)$.
  \end{description}
\end{defn}
\begin{rem}
  \label{rem:macshane_mfGH}
  Let $(X_n,\mu_n,p_n,\Phi_n)\mfghcvj (Y,\nu,q,\Psi)$ in the container
  $(Z,z)$. By replacing $Z$ with $Z\times \real$ and slightly shifting
  basepoints we may assume:
  \begin{equation}
    \label{eq:macshane_mfGH_1}
    \begin{aligned}
      X_n\cap X_m&=\emptyset\quad(n\ne m)\\
      Y\cap X_n&=\emptyset\quad(\forall n).
    \end{aligned}
  \end{equation}
  Then we might try to define $\varphi_{Z,k}=\varphi_{n,k}$ on $X_n$
  and $\varphi_{Z,k}=\psi_k$ on $Y$. This yields a continuous function
  on $\bigcup_nX_n\cup Y$ which one might extend to $Z$. However, fix
  $R>0$; then $\varphi_{Z,k}|B(z,R)$ is almost $1$-Lipschitz up to
  additive errors that depend on the Hausdorff-distance between
  $X_n\cap B(z,R)$ and $Y\cap B(z,R)$. By replacing $Z$ with
  $Z\times\real$, passing to a subsequence and shifting basepoints, one
  can verify the following lemma.
\end{rem}
\begin{lem}
  \label{lem:mfgh_macshane}
  Let $(X_n,\mu_n,p_n,\Phi_n)\mfghcvj (Y,\nu,q,\Psi)$ in the container
  $(Z,z)$. Up to passing to a subsequence and taking a new container
  of the form $(Z\times\real,(z,0))$, one can assume that for each
  $k\le\#\Psi$ there is a $1$-Lipschitz function
  $\varphi_{Z,k}:Z\to\real$ such that:
  \begin{enumerate}
  \item $\varphi_{Z,k}|Y=\psi_k$;
  \item for each $R>0$, $k\le\#\Psi$, there is an $N(R,k)$ such that,
    if $n\ge N(R,k)$, one has:
    \begin{equation}
      \label{eq:mfgh_macshane_s1}
      \varphi_{Z,k}|B(z,R)\cap X_n = \varphi_{n,k}.
    \end{equation}
  \end{enumerate}
\end{lem}
For $\lambda>0$ and a metric space $X$, we denote by $\lambda X$ the
metric space where the metric on $X$ has been rescaled by $\lambda$,
i.e.~$d_{\lambda X}=\lambda d_X$. Let $\Phi$ be a countable collection
of $1$-Lipschitz functions on $X$ and $p\in X$. Then we denote by
$\Phi_{\lambda,p}$ the collection of $1$-Lipschitz functions on
$\lambda X$ given by:
\begin{equation}
  \label{eq:resc_collection}
  \Phi_{\lambda,p}=\left\{
    \lambda\left(\varphi-\varphi(p)\right): \varphi\in\Phi
  \right\}.
\end{equation}
\begin{defn}[Blow-ups]
  \label{defn:blow_ups}
  A \textbf{blow-up} of $X$ at a point $p$ is a pointed metric space
  $(Y,q)$ such that for some sequence $\lambda_n\nearrow\infty$:
  \begin{equation}
    \label{eq:blow_ups_1}
    (X_n,p_n)=(\lambda_n X,p)\ghcvj (Y,q);
  \end{equation}
  in this case we say that $(Y,q)$ is \textbf{realized} by the sequence
  of rescalings $\{\lambda_n\}$. The set of blow-ups of $X$ at $p$ will
  be denoted by $\tang(X,p)$.
  \par\noindent 
  {Let $(X,\mu)$ be a metric measure space, $p\in X$ a
  basepoint and $\lambda>0$ a dilating factor; we define the
  normalization constant $c_\mu(p,\lambda)$ for the unit ball of
  $\lambda X$ centred at $p$ as follows:
  \begin{equation}
    \label{eq:blow_ups_0}
    c_\mu(p,\lambda) = \left(
      \int_{B_X(p,\lambda^{-1})}(1-\lambda d_X(p,x))\,d\mu(x)
    \right)^{-1}.
  \end{equation}}
{Note that in~(\ref{eq:blow_ups_0})
  we used the subscript $X$ to highlight that balls are taken with
  respect to the metric $d_X$ of $X$.}
A \textbf{blow-up} of $(X,\mu)$ at a point $p$ is a measure-normalized pointed metric space
  $(Y,\nu,q)$ such that for some sequence $\lambda_n\nearrow\infty$:
  \begin{equation}
    \label{eq:blow_ups_2}
    (X_n,\mu_n,p_n)=(\lambda_n X,
    {c_\mu(p,\lambda_n)\,\mu},p)\mghcvj (Y,\nu,q);
  \end{equation}
  in this case we say that $(Y,\nu,q)$ is \textbf{realized} by the sequence
  of rescalings $\{\lambda_n\}$. The set of blow-ups of $X$ at $p$ will
  be denoted by $\tang(X,\mu,p)$.
  \par\noindent A \textbf{blow-up} of $(X,\mu,\Phi)$ at a point $p$ is a function space
  $(Y,\nu,q,\Psi)$ such that for some sequence $\lambda_n\nearrow\infty$:
  \begin{equation}
    \label{eq:blow_ups_3}
    (X_n,\mu_n,p_n,\Phi_n)=(\lambda_n X,
    {c_\mu(p,\lambda_n)\,\mu},,p,\Phi_{\lambda_n,p})\mfghcvj (Y,\nu,q,\Psi);
  \end{equation}
  in this case we say that $(Y,\nu,q,\Psi)$ is \textbf{realized} by the sequence
  of rescalings $\{\lambda_n\}$. The set of blow-ups of $X$ at $p$ will
  be denoted by $\tang(X,\mu,p,\Phi)$.
\end{defn}
The following theorem summarizes variants in the metric setting of
Preiss' phenomenon that tangents of tangents are tangents
\cite{preiss_geometry_measures}; (1) is due to
\cite{ledonne_tangents}, (2) to \cite{gigli_weak_euclidean_tangs}; the
proof of (3) is omitted as can be easily reconstructed from
\cite{gigli_weak_euclidean_tangs}. It is clear that (3) can be
generalized in further directions, e.g.~in the context of blowing-up
pseudodistances. 
\begin{thm}(Shifted rescalings of blow-ups are blow-ups)
  \label{thm:blow_ups_are_blow_ups}
  Let $(X,\mu)$ be a doubling metric measure space and $\Phi$ a
  countable collection of $1$-Lipschitz functions on $X$. Then for
  $\mu$-a.e.~$p\in X$ the following holds:
  \begin{enumerate}
  \item For each $(Y,q)\in\tang(X,p)$, for any
    $(\lambda,q')\in(0,\infty)\times Y$ one has $(\lambda
    Y,q')\in\tang(X,p)$; in particular,
    $\tang(Y,q')\subset\tang(X,p)$;
  \item For each $(Y,\nu,q)\in\tang(X,\mu,p)$, for any
    $(\lambda,q')\in(0,\infty)\times Y$ one has $(\lambda
    Y,
    {c_\nu(q',\lambda)\,\nu,q'})\in\tang(X,\mu,p)$; in
    particular, $\tang(Y,\nu,q')\subset\tang(X,\mu,p)$;
  \item For each $(Y,\nu,q,\Psi)\in\tang(X,\mu,p,\Phi)$, for any
    $(\lambda,q')\in(0,\infty)\times Y$ one has $(\lambda
    Y,
    {c_\nu(q',\lambda)\,\nu,q'}),q',\Psi_{\lambda,q'})\in\tang(X,\mu,p,\Phi)$;
    in particular, $\tang(Y,\nu,q',\Psi)\subset\tang(X,\mu,p,\Phi)$; 
  \end{enumerate}
\end{thm}
\subsection{Blow-up of Weaver derivations and weak convergence for
  normal currents}
\label{subsec:bwp_wea_norm}
In this subsection we analyze how rescalings affect the modules
$\wder\mu.$ and $\wform\mu.$. When we want to highlight an object that
refers to the rescaled
space $\lambda X$ (resp.~the original space $X$), we add $\lambda X$
(resp.~$X$) to the notations. Recall that in this paper we use the
notation $\glip f.$ for the global 
Lipschitz constant of $f$. We then sketch the details of how to use a
result in \cite[Sec.~7]{cks_metric_diff} to blow-up Weaver derivations.
The following lemma is elementary and the proof is omitted.
\begin{lem}
  \label{lem:fun_transf}
  Let $X$ be a metric space, $\lambda>0$ and $f$ a real-valued
  Lipschitz function defined on $X$. Then $\glipdec f,X.=C$ if and
  only if $\glipdec \lambda f,\lambda X.=C$, and $\glipdec
  \lambda^{-1}f,X.=C$ if and only if $\glipdec f,\lambda X.=C$.
\end{lem}
\begin{defn}[Rescaling of an Alberti representation]
  \label{defn:alb_rep_resc}
  Let $\albrep.=[Q,w]$ be an Alberti representation of the measure
  $\mu$ on $X$; let $\lambda>0$ and define
  $\rescf_{\lambda}:\frags(X)\to\frags(\lambda X)$ as follows:
  \begin{equation}
    \label{eq:alb_rep_resc_1}
    \begin{aligned}
      \dom\left(
        \rescf_\lambda(\gamma)
      \right) &= \lambda\dom\gamma\\
      \rescf_\lambda(\gamma)(t)&=\gamma(t/\lambda).
    \end{aligned}
  \end{equation}
  Let $p\in X$ and define $\rescf_{\lambda,p}:\frags(X)\to\frags(\lambda X)$ as follows:
  \begin{equation}
    \label{eq:alb_rep_resc_2}
    \begin{aligned}
      \dom\left(
        \rescf_{\lambda,p}(\gamma)
      \right) &= \lambda(\dom\gamma-s_{\gamma,p})\\
      \rescf_{\lambda,p}(\gamma)(t)&=\gamma(t/\lambda+s_{\gamma,p}),
    \end{aligned}
  \end{equation}
  where $s_{\gamma,p}$ is such that:
  \begin{equation}
    \label{eq:alb_rep_resc_3}
    d\left(
      \gamma(s_{\gamma,p}),p
    \right) = \min_{t\in\dom\gamma}
    d\left(
      \gamma(t),p
    \right).
  \end{equation}
  Note that a measurable choice of $\rescf_{\lambda,p}$ can be
  obtained via a measurable selection principle. The rescalings
  $\albrep\lambda.$ and $\albrep \lambda,p.$ are defined as follows:
  \begin{equation}
    \label{eq:alb_rep_resc_4}
    \begin{aligned}
      \albrep\lambda.&=[\rescf_{\lambda}{}_{\#}Q,w]\\
      \albrep\lambda,p.&=[\rescf_{\lambda,p}{}_{\#}Q,w].
    \end{aligned}
  \end{equation}
  Note that if $\albrep.$ is $[C_0,D_0]$-biLipschitz on $X$, then
  $\albrep\lambda.$ and $\albrep\lambda,p.$ are
  $[C_0,D_0]$-biLipschitz on $\lambda X$.
\end{defn}
\def\ldec{^{(\lambda X)}}
\def\xdec{^{(X)}}
\pcreatenrm{wflx}{\wform\mu;\lambda X.}{|}
\pcreatenrm{wfx}{\wform\mu;X.}{|}
\begin{lem}
  \label{lem:locnrm_transf}
  Let $f\in\lipalg X.$; then
  \begin{equation}
    \label{eq:locnrm_transf_s1}
    \lambda^{-1}\wfxnrm d\xdec f.=\wflxnrm d\ldec f..
  \end{equation}
\end{lem}
\begin{proof}
  Assume that $\wfxnrm d\xdec f.\ge\alpha$ on a set $S$. Fix
  $\varepsilon>0$; then by 
  {Theorem~\ref{thm:char_loc_norm}} one can
  find a $(1,1+\varepsilon)$-biLipschitz Alberti representation
  $\albrep S.$ of $\mu\on S$ with $f$-speed $\ge\alpha-\varepsilon$.
  This means that for $Q$-a.e.~$\gamma$ and
  $\lebmeas.\on\dom\gamma$-a.e.~$t$ one has:
  \begin{equation}
    \label{eq:locnrm_transf_p1}
    (f\circ\gamma)'(t)\ge(\alpha-\varepsilon).
  \end{equation}
  Now the rescaled representation $\albrep S,\lambda.$ gives a
  $(1,1+\varepsilon)$-biLipschitz (wrt.~the metric on $\lambda X$)
  Alberti-representation of $\mu\on S$ with a lower bound on the
  $f$-speed, that can be obtained using:
  \begin{equation}
    \label{eq:locnrm_transf_p2}
    \begin{split}
      (f\circ\rescf_\lambda(\gamma))'(t)&=\frac{d}{dt}\left(
        f(\gamma(t/\lambda))
      \right)\\ &=\lambda^{-1}(f\circ\gamma)'(t/\lambda)\ge\lambda^{-1}(\alpha-\varepsilon)\\
      &\ge\lambda^{-1}(1+\varepsilon)^{-1}(\alpha-\varepsilon)\metdiff\rescf_\lambda(\gamma)(t),
    \end{split}
\end{equation}
  where the metric differential $\rescf_\lambda(\gamma)(t)$ is
  computed with respect to the metric on $\lambda X$. Thus:
  \begin{equation}
    \label{eq:locnrm_transf_p3}
    \wflxnrm d\ldec f.\ge\lambda^{-1}\alpha \quad\text{$\mu$-a.e.~on $S$.}
  \end{equation}
\end{proof}
\def\bdec{_{\text{\normalfont blow}}}
\def\ndec{^{(n)}}
\def\idec{^{(\infty)}}
\pcreatenrm{wd}{\wder\mu.}{|}
\pcreatenrm{wf}{\wform\mu.}{|}
\pcreatenrm{wfy}{\wform\nu.}{|}
\begin{thm}[Blow-up of Weaver derivations]
  \label{thm:weaver_blowup}
  Let $(X,\mu)$ be a complete separable doubling metric measure space where
  $\mu$ is an asymptotically doubling Radon measure. Let $\Phi$ be a countable
  collection of real-valued $1$-Lipschitz functions on $X$.
  Let $D\in\wder \mu;X.$ be of the form $D=D_{\albrep.}$ where $\albrep.$ is a
  $[C_0,D_0]$-biLipschitz Alberti representation of $\mu$. Then there
  is a $\mu$-full measure Borel set $X\bdec$ such that for each $p\in
  X\bdec$ and each $(Y,\nu,q,\Psi)\in\tang(X,\mu,p,\Phi)$ one can
  blow-up $D$ as follows. Assume that $(Y,\nu,q,\Psi)$ is realized by some
  sequence of rescalings $\{\lambda_n\}_n$ and let
  $\albrep n.=\albrep \lambda_n,p.$ and $D\ndec=D_{\albrep n.}$ be
  the corresponding derivation in $\wder\mu_n;X_n.$. Then there is an Alberti
  representation $\albrep\infty.=[Q_\infty,1]$ of $\nu$ such that, defining
  $D\idec=D_{\albrep\infty.}$, the following holds:
  \begin{description}
  \item[(Bw-up1)] $\spt Q_\infty$ consists of lines in $Y$ with
    constant speed in $[C_0,D_0]$;
  \item[(Bw-up2)] For each $\varphi\in\Phi$ and $\gamma$ in $\spt
    Q_\infty$ there is a $c_{\varphi,\gamma}\in[-1,1]$ such that, if
    $\psi\in\Psi$ denotes the corresponding blow-up of $\varphi$:
    \begin{equation}
      \label{eq:weaver_blowup_s1}
      (\psi\circ\gamma)'(t)=c_{\varphi,\gamma}\metdiff\gamma(t)\quad(\forall t\in\real);
    \end{equation}
  \item[(Bw-up3)] Suppose that
    $(X_n,\mu_n,x_n,\Phi_{\lambda_n,x})\mfghcvj(Y,\nu,y,\Psi)$ in the
    container $(Z,z)$; then up to passing to a subsequence (which
    might depend on the container) one can assume that for each
    $(g,f)\in\ccompact Z.\times \lipalg Z.$ one has:
  \end{description}
  \begin{equation}
    \label{eq:weaver_blowup_s2}
    \lim_{n\to\infty}\int g\,D\ndec f\,d\mu_n=\int g\,D\idec f\,d\nu.
  \end{equation}
\end{thm}
\begin{rem}
  \label{rem:curren_reform}
The statement~(\ref{eq:weaver_blowup_s2}) has a cleaner interpretation
in the language of metric currents \cite{curr_alb}: the metric
currents:
\begin{equation}
  \label{eq:curr_reform_1}
  T_n(g,f)=\int g\,D\ndec f\,d\mu_n
\end{equation}
are converging to the normal current:
\begin{equation}
  \label{eq:curr_reform_2}
  T(g,f)=\int g\,D\idec f\,d\nu
\end{equation}
in the weak topology.
\end{rem}
\begin{rem}
  \label{rem:porous_nullity}
  One can also remove from Theorem~\ref{thm:weaver_blowup} the
  assumption that $X$ is doubling. In that case one has 
  to replace $X$ by an appropriate full-measure Borel $\tilde{X}\subset X$ and blow-up
  $\tilde{X}$. However, note that if porous sets are not $\mu$-null,
  there will be a set of positive measure where $\tilde{X}$ and $X$ do
  not have the same blow-ups (e.g.~it might happen that one cannot
  apply Gromov's compactness Theorem to $X$). However, if one uses ultrafilters, one can show
  that if $(Y,y)$ is a blow-up of $X$ at $x\in \tilde{X}$, then there
  is a blow-up $(\tilde{Y},\nu,y)$ of $(\tilde{X},\mu\on\tilde{X})$ at
  $x$ such that $\tilde{Y}\subset Y$.
\end{rem}
\def\pars{{\rm PAR}(\varepsilon{},S)}%
\def\fillfrag{\text{\normalfont Fill}}
\def\rprm{\text{\normalfont Rep}}
\begin{proof}
  The proof can be reconstructed from the argument in
  \cite[Sec.~7]{cks_metric_diff} where the result is stated in a less
  general context: $(X,\mu)$ is a differentiability space and $\Phi$
  is a finite set of Lipschitz functions. The only item that requires
  further justification is \textbf{(Bw-up3)}. We highlight the
  additional arguments; we will refer to the notation and setting in
  \cite[Sec.~7]{cks_metric_diff}; in particular, $r_n=\lambda_n^{-1}$ and
  for $\gamma\in\frags(X)$ we define:
  \begin{equation}
    \label{eq:weaver_blow_up_p1}
    \Psi(\gamma)=\mpush\gamma.(\lebmeas.\on\dom\gamma).
  \end{equation}
  {However note that we are using a 
    convention to normalize the measures different from that
    in~\cite{cks_metric_diff}, where measures are
    rescaled to give unit mass to the open unit ball. This essentially
    amounts to replacing some terms in~\cite[Sec.~7]{cks_metric_diff}
    of the form $\mu'(B(x,r))$ with $1/c_{\mu'}(x,r^{-1})$, see the following
    discussion for more details.}
  In \cite[Sec.~7]{cks_metric_diff} we have performed some preliminary
  steps:
  \begin{description}
  \item[(PS1)] If $\albrep.=[Q,w]$, instead of working with $\mu$, we
    work with the measure $\mu'$ corresponding to the Alberti
    representation $\albrep.=[Q,1]$. This can be done as $\mu$ is
    asymptotically
    doubling and as $D_{\albrep.}$ is not affected on the set where $w\ne0$;
  \item[(PS2)] For $\gamma\in\spt Q$ we have split $\Psi(\gamma)$ in a
    regular part $\Psi_{\pars}(\gamma)$ and an irregular part
    $\Psi^c_{\pars}(\gamma)$ (compare equation \sync
    eq:cks_reg_dec.\ in \cite[Sec.~7]{cks_metric_diff}). There
    are straightforward modifications for the regularity requirements
    to handle a countable collection $\Phi$; here $\varepsilon$ and
    $S$ play the r\^ole of parameters selecting ($S$) how long
    $\gamma$ is, and ($\varepsilon$) how
    close $\gamma$ is to a constant speed segment on which the first $\lfloor S\rfloor$
    elements of $\Phi$ are close to affine maps;
    \def\pars{{\rm PAR}(\varepsilon_m{},S_m)}%
  \item[(PS3)] We have chosen an $l^1$-sequence
    $\{\varepsilon_m\}\subset(0,\infty)$ and used
    (Lemma \sync lem:cks_reg_dec_borel.\ in \cite[Sec.~7]{cks_metric_diff}) measure-differentiation to find a
    Borel set $U$ and scales $s_m\le S_m$ such that:
    \begin{equation}
      \label{eq:weaver_blow_up_p2}
      \begin{aligned}
        \mu(X\setminus U)&\le\sum_{m=1}^\infty\varepsilon_m\\
        \mu'^c_{\pars}\left(
          B(x,r)
        \right)
        &\le\frac{\varepsilon_m}{
          {c_{\mu'}(x,r^{-1})}}\quad(\forall
        \,x\in U,r\le s_m),
      \end{aligned}
    \end{equation}
    where
    \begin{equation}
      \label{eq:weaver_blow_up_p3}
      \begin{aligned}
        \mu'_{\pars}&=\int\Psi_{\pars}(\gamma)\,dQ(\gamma)\\
        \mu'^c_{\pars}&=\int\Psi^c_{\pars}(\gamma)\,dQ(\gamma).
      \end{aligned}
    \end{equation}
    {Note that the second equation in~(\ref{eq:weaver_blow_up_p2}) differs from the
    corresponding one in Lemma \sync lem:cks_reg_dec_borel.\ of
    \cite[Sec.~7]{cks_metric_diff}, as we have replaced
    {$\mu'\left(
          B(x,r)
        \right)$} with $1/c_{\mu'}(x,r^{-1})$. This is possible as
      $\mu'$ is asymptotically doubling and the function
      $1-r^{-1}d(\cdot,x)$ has value at least $\frac{1}{2}$ on
      $B(x,r/2)$; this also implies that for each fixed $R_0\ge 1$ one can
      find $C(R_0)$ such that
      $c_{\mu'}(x,r^{-1})\lesssim_{C(R_0)}c_{\mu'}(x,R_0^{-1}r^{-1})$:
      we will not insist on this point any further;}
  \item[(PS4)] We take $x\in U$ and, up to enlarging $(Z,z)$, we
    assume that $(Z,z)$ is a Banach space.
  \end{description}
  \def\pars{{\rm PAR}(\varepsilon_m{},S_m)}%
  We now fix $R_0$ large enough so that, given $(g,f)\in\ccompact
  Z.\times \lipalg Z.$, $\spt g\subset B(z,R_0)$.
  As $|D_{\albrep n.}f|\le D_0\glip f.$, we conclude from equation
  \sync eq:alberti_blow_up_psix.\ in \cite[Sec.~7]{cks_metric_diff} that:
    \def\stddeno{c_{\mu'}(p,r_n^{-1})}
  \begin{equation}    
    \label{eq:weaver_blow_up_p4}
    \begin{split}
    \left|
      {\stddeno}\int_{\sball X_n,p,R_0.}g
      D_{\albrep n.}f\,d\mu'^c_{\pars}\right|&\le\|g\|_\infty D_0\glip
    f.
    {\stddeno}\\
    &\times\mu'^c_{\pars}\left(\sball
        X,p,r_nR_0.\right);
  \end{split}
\end{equation}
  note that we are regarding $\mu'^c_{\pars}$ as a measure on $X_n$
  and are suppressing from the notation the isometric embeddings in
  $Z$.
  \par As $Z$ is a Banach space, one can introduce a ``filling map''
  $\fillfrag$ which fills fragments to curves (details are discussed
  in \cite[Sec.~7]{cks_metric_diff}). Using the Hausdorff topology on fragments
  one can introduce reparametrization maps $\{\rprm_{n}\}_n$ which agree, up
  possibly to a translation, with the map $\{\rescf_{\lambda_n,x}\}_n$.
  \def\gamman{\Gamma_{X_n}}%
  \def\gammatn{\tilde\Gamma_{X_n}}%
    \def\pars{{\rm PAR}(\varepsilon_m{},S_m/r_n)}%
\par Let $\gamman\subset\frags(X)$ denote the set of those
$[C_0,D_0]$-biLipschitz fragments which
intersect $\scball X,x,2R_0r_n.$, and by $\gammatn\subset\gamman$ the Borel subset of those $\gamma$
such that:
\begin{equation}
  \label{eq:weaver_blow_up_p5}
  \chi_{\sball
    X_n,x,R_0.}\Psi_{\pars}\left(\rprm_n(\gamma)\right)\ne0.
\end{equation}
Then combining equations \sync eq:cks_alberti_blow_up_esttwo.\ and  
\sync eq:cks_alberti_blow_up_estthree.\ in \cite[Sec.~7]{cks_metric_diff} with the definition
of $D\ndec$ one arrives at the estimate:
\begin{equation}
  \label{eq:weaver_blow_up_p6}
  \begin{split}
  \biggl|
    \int_{X_n}gD\ndec f\,d\mu'_n &-
    {c_{\mu'}(p,r_n^{-1})}
    \int_{\gammatn}dQ(\gamma)\int r_n
    g\circ(\fillfrag\circ\rprm_n)(\gamma)\,\\
    &\times
    \left(
      f\circ(\fillfrag\circ\rprm_n)(\gamma)
    \right)'\chi_{B_{X_n}(x,R_0)}\,d\Psi\left(
      \fillfrag\circ\rprm_n(\gamma)
    \right)
    \biggr|\\
    &\le \|g\|_\infty D_0\glip f.\biggl(
      O(\varepsilon_m)+C(C_0,D_0)\varepsilon_m\\
      &\mskip 18mu\times 
      {c_{\mu'}(p,r_n^{-1}){\mu'(B_X(p,2r_nR_0))}}
      \biggr).
\end{split}
\end{equation}
{Note that by using the property of $\mu'$ being asymptotically doubling one can find a
uniform bound on $c_{\mu'}(p,r_n^{-1}){\mu'(B_X(p,2r_nR_0))}$ which
  depends on $R_0$ but not on $n$.}
Letting:
\begin{equation}
  \label{eq:weaver_blow_up_p7}
  Q_n=
  {c_{\mu'}(p,r_n^{-1})}\,r_n\mpush \fillfrag\circ\rprm_n(\gamma).Q\on\gammatn
\end{equation}
one can use the mass estimate
(equation \sync eq:cks_alberti_blow_up_pelevenplusone.\ in \cite[Sec.~7]{cks_metric_diff})
to show that $Q_n\wkcvj Q_{R_0}$ by passing to a subsequence. By a
diagonal argument (compare the proof of
Lemma \sync lem:local_alberti_blow_up.\ in \cite[Sec.~7]{cks_metric_diff}) which involves
$R_0\nearrow\infty$ one can obtain a limit $Q_\infty$ of the $Q_{R_0}$
and an Alberti representation $\albrep\infty.=[Q_\infty,1]$. To
prove~(\ref{eq:weaver_blowup_s2}) it suffices to use the definition of
the weak* topology for Radon measures by showing that if $\Omega$ is a
closed subset of $\curves(Z)$ with all elements of $\Omega$ having
their domain contained in a given bounded interval of $\real$, and
with $\sup_{\Omega\ni\gamma}\glip\gamma.<\infty$, then the map:
\begin{equation}
  \label{eq:weaver_blow_up_p8}
  \Omega\ni\gamma\mapsto\int g\circ\gamma (f\circ\gamma)' d\Psi(\gamma)
\end{equation}
is continuous. This reduces to the weak* continuity of the derivation
$\partial_x\in\wder{\lebmeas.}.$ on the real line \cite{weaver00}.
\end{proof}
\begin{cor}
  \label{cor:weaver_basis_blow}
  Let $(X,\mu)$ be a metric measure space with $\mu$ asymptotically
  doubling and assume that $\wform\mu.$ is free on
  $\{d\varphi_i\}_{i=1}^N$ where each $\varphi_i$ is $1$-Lipschitz,
  and let $\Phi=\{\varphi_i\}_{i=1}^N$. Then for $\mu$-a.e.~$x$ and for
  each $(Y,\nu,y,\Psi)\in\tang(X,\mu,x,\Phi)$ the following holds:
  \begin{description}
  \item[(FormBlow1)] The submodule of $\wform\nu.$ generated by the
    $\{d\psi_i\}_{i=1}^N$ (where $\{\psi_i\}_{i=1}^N=\Psi$) is free;
  \item[(FormBlow2)] For each $a\in\real^N$ there is an Alberti
    representation $[Q_a,1]$ of $\nu$, where $Q_a$ is concentrated on
    the set of unit-speed lines of $Y$ satisfying:
    \begin{equation}
      \label{eq:weaver_basis_blow_s1}
      \sum_{i=1}^Na_i\left(
        \psi_i\circ\gamma(t)
        -
        \psi_i\circ\gamma(s)
      \right) = \wfnrm
      \sum_{i=1}^Na_id\varphi_i(x)
      .(t-s)\quad(\forall\, t\ge s);
    \end{equation}
  \item[(FormBlow3)] Moreover, if $(X,\mu)$ is a differentiability
    space, one also has:
    \begin{equation}
      \label{eq:weaver_basis_blow_s2}
      \wfnrm
      \sum_{i=1}^Na_id\varphi_i(x)
      . =
      \biglip\left(
        \sum_{i=1}^Na_i\varphi_i
      \right)(x) =
      \glip\sum_{i=1}^Na_i\psi_i..
    \end{equation}
  \end{description}
\end{cor}
\begin{proof}
  \par\noindent\texttt{Step1: Blowing-up a single function $f$.}
  \par Assume that $f$ is $1$-Lipschitz with $df\ne0$ $\mu$-a.e.; by
  Theorem~\ref{thm:char_loc_norm}, for each $n$, $\mu$ admits a
  $(1,1+\frac{1}{n})$-biLipschitz Alberti representation $\albrep n.$
  with $f$ speed $\ge\wfnrm df.-1/n$; note also that $\wfnrm df.$ is
  necessarily an upper bound on the $f$-speed of any Alberti representation. We use
  Theorem~\ref{thm:weaver_blowup} to find a $\mu$-full measure subset
  $U$ where one can blow-up each $\albrep n.$ and on which $\wfnrm
  df.$ is approximately continuous. For a Lebesgue point $x$ of $U$
  let $(Y,\nu,y,\{g\})\in\tang(X,\mu,x,\{f\})$, and denote by $\albrep
  n,\infty.=[Q_n,1]$ the corresponding blow-up of $\albrep
  n.$. Then $Q_n$ is concentrated on the set of lines in $Y$ with
  constant speed in $[1,1+1/n]$ and which satisfy:
  \begin{multline}
    \label{eq:weaver_basis_blow_p1}
    g(\gamma(t)) -
    g(\gamma(s)) \in \left[
      \left(
        \wfnrm df(x). - \frac{1}{n}
      \right)(t-s),
      \left(1+\frac{1}{n}\right)
        \wfnrm df(x).
      (t-s)
    \right]\\
    (\forall\, t\ge s).
  \end{multline}
  By a compactness argument one obtains an Alberti representation
  $\albrep \infty.=[Q_\infty,1]$ of $\nu$ where $Q_\infty$ is
  concentrated on the set of unit speed lines of $Y$ satisfying:
  \begin{equation}
    \label{eq:weaver_basis_blow_p2}
    g(\gamma(t))
    -
    g(\gamma(s)) = \wfnrm df(x).(t-s)\quad(\forall\, t\ge s).
  \end{equation}
  \par\noindent\texttt{Step 2: Approximation by rational
    combinations.}
  \par We apply \texttt{Step 1} to each function
  $\Phi_\rational=\left\{\sum_{i=1}^Na_i\varphi_i\right\}_{a\in\rational^N}$,
  and one also assumes that $x$ is an approximate continuity point of each map:
  \begin{equation}
    \label{eq:weaver_basis_blow_p3}
    \tilde{x}\mapsto\wfnrm
    \sum_{i=1}^Na_id\varphi_i(\tilde{x})
    .\quad(a\in\rational^N).
  \end{equation}
  To obtain \textbf{(FBlow1)}, \textbf{(FBlow2)} one then uses the
  fact that $\rational^N$ is dense in $
  {\real^N}$ and the linearity
  of blow-ups, i.e.~that the blow-up of $\sum_{i=1}^Na_i\varphi_i$ is
  $\sum_{i=1}^Na_i\psi_i$. \textbf{(FBlow3)} follows from the
  quantitative characterization of differentiability spaces
  (\textbf{(Schioppa)} in Theorem~\ref{thm:diff_summary}) and from the fact that given a Lipschtiz
  function $f$, at $\mu$-a.e.~$x$ any blow-up of $f$ at $x$ has
  Lipschitz constant at most $\biglip f(x)$.
\end{proof}
\section{A local approach to fail differentiability}
\label{sec:loc_app}
In this section we discuss a methodology for the gluing part of the
argument. Even though we end up using ``cubical'' tiles in
Section~\ref{sec:diff_blow}, other geometries for tiles might be
helpful in deducing other properties of blow-ups.%
\def\vdec{_{\text{\normalfont var}}}
\def\xvdec{_{x,\text{\normalfont var}}}
The following theorem summarizes 
{sufficient} conditions to show that
a metric measure space is not a differentiability space; we follow an
approach motivated by Weaver derivations; for the proof see
\cite[Sec.~5.3]{deralb} (or \cite{bate_thesis_final} for an approach
via charts).
\begin{thm}
  \label{thm:indep_functions}
  Let $(X,\mu)$ be a metric measure space with $\mu(X)=1$.
  Assume that for some $(\alpha\vdec,L)\in(0,\infty)^2$ for each
  $\varepsilon>0$ there is an $L$-Lipschitz function $f_\varepsilon$
  such that:
  \begin{description}
  \item[(SmDiff)] $\mu\left(
      \left\{
        x\in X: \wfnrm df_\varepsilon(x).>\varepsilon
      \right\}
    \right)\le\varepsilon$;
  \item[(PosVar)] There is a Borel set $X\vdec\subset X$ with
    $\mu(X\setminus X\vdec)\le\varepsilon$, and such that for each
    $x\in X\vdec$ there is an $x\vdec=x\vdec(x)$ satisfying:
    \begin{align}
      \label{eq:indep_function_s1}
      d(x,x\vdec)&\in(0,\varepsilon]\\
      \label{eq:indep_function_s2}
      \left|
        f_\varepsilon(x) - f_\varepsilon(x\vdec)
      \right| &\ge\alpha\vdec d(x,x\vdec).
    \end{align}
  \end{description}
  Then $(X,\mu)$ is not a differentiability space.
\end{thm}
\def\gdec{_{\text{\normalfont grad}}}
\def\lsdec{_{\text{\normalfont loss}}}
\begin{defn}[Tiles]
  \label{defn:tile}
  Let
  $(L,\alpha\vdec,\varepsilon\gdec,\varepsilon\lsdec,r)\in(0,\infty)^5$;
  an
  \textbf{$[L,\alpha\vdec,\varepsilon\gdec,\varepsilon\lsdec,r]$-tile}
  at $x$ is a pair $(S_x,f_x)$ such that:
  \begin{description}
  \item[(T1)] $S_x$ is 
    {closed}, $S_x\subset\ball x,r.$ and $\diam S_x\approx_L r$;
  \item[(T2)] $\mu(S_x)\gtrsim_L\mu\left( \ball x,r. \right)$;
  \item[(T3)] $f_x$ is $L$-Lipschitz and for each $p\in S_x$ one has:
    \begin{equation}
      \label{eq:tile_1}
      \left|
        f_x(p)
      \right| \le L d(p,S_x^c);
    \end{equation}
  \item[(T4)] $\mu\left(
      \left\{
        p\in S_x: \wfnrm df_x(p).>\varepsilon\gdec
      \right\}
    \right)\le\varepsilon\lsdec\mu(S_x)$;
  \item[(T5)] There is a Borel set $S\xvdec\subset S_x$ such that
    $\mu(S_x\setminus S\xvdec)\le\varepsilon\lsdec\mu(S_x)$ and for
    each $y\in S\xvdec$ there is a $y\vdec=y\vdec(y)$ such
    that:
    \begin{align}
      \label{eq:tile_2}
      d(y,y\vdec)&\in(0,\varepsilon\lsdec]\\
      \label{eq:tile_3}
      \left|
        f_x(y) - f_x(y\vdec)
      \right| &\ge\alpha\vdec d(y,y\vdec).
    \end{align}
  \end{description}
\end{defn}
\begin{thm}
  \label{thm:glue_tiles}
  Let $(X,\mu)$ be a complete metric measure space with $\mu(X)=1$, and such that
  the Vitali Covering Lemma holds for $\mu$. Assume that there are
  constants
  $(L,\alpha\vdec,\varepsilon\gdec,\varepsilon\lsdec)\in(0,\infty)^4$
  such that for $\mu$-a.e.~$x\in X$ there is a sequence of scales
  $\{r_n=r_n(x)\searrow0\}_n$ such that for each $n$ there is an
  $[L,\alpha\vdec,\varepsilon\gdec,\varepsilon\lsdec,r_n]$-tile at
  $x$. Then the assumption of Theorem~\ref{thm:indep_functions} holds
  with $L$ replaced by $2L$, with $\alpha\vdec$ replaced by $\alpha\vdec/2$, and with
  $\varepsilon=\max(\varepsilon\gdec,\varepsilon\lsdec)$. 
\end{thm}
\begin{proof}
  By \textbf{(T1)}, \textbf{(T2)} tiles are 
  {closed}
  and comparable to balls in
  measure and shape; thus by the Vitali Covering Lemma we can find
  tiles $\{(S_{x_i},f_{x_i})\}_i$ such that the sets $\{S_{x_i}\}$ are
  pairwise disjoint and $\mu(X\setminus\bigcup_iS_{x_i})=0$. We let
  $f=f_{x_i}$ on each $S_{x_i}$ and $f=0$ on
  $X\setminus\bigcup_iS_{x_i}$.
  \par\noindent\texttt{Step 1: $f$ is $2L$-Lipschitz.}
  \par We will use \textbf{(T3)} to compare values of $f$ at
  points belonging to different tiles. If $x,y\in S_{x_i}$ we have:
  \begin{equation}
    \label{eq:glue_tiles_p1}
    \left|
      f(x) - f(y)
    \right| =
    \left|
      f_{x_i}(x) - f_{x_i}(y)
    \right| \le Ld(x,y)
  \end{equation}
  because $f_{x_i}$ is $L$-Lipschitz. If $x\in S_{x_i}$ and $y\in
  S_{x_j}$ for $i\ne j$:
  \begin{equation}
    \label{eq:glue_tiles_p2}
    \begin{split}
      \left|
        f(x) - f(y)
      \right| &\le
      \left|
        f_{x_i}(x)
      \right| +
      \left|
        f_{x_j}(y)
      \right| \le
      L \left[
        d(x,S_{x_i}^c) + d(y,S_{x_j}^c)
      \right]\\ &\le
      2Ld(x,y).
    \end{split}
  \end{equation}
  If $x\in S_{x_i}$ and $y\in X\setminus\bigcup_i S_{x_i}$ we have:
  \begin{equation}
    \label{eq:glue_tiles_p3}
    \left|
      f(x) - f(y)
    \right| =
    \left|
      f_{x_i}(x)
    \right| \le Ld(x,S_{x_i}^c)\le Ld(x,y).
  \end{equation}
  If $x,y\in X\setminus\bigcup_iS_{x_i}$:
  \begin{equation}
    \label{eq:glue_tiles_p4}
    \left|
      f(x) - f(y)
    \right| = 0.
  \end{equation}
  \par\noindent\texttt{Step 2: \textbf{(SmDiff)} holds with
    $\varepsilon=\max\{\varepsilon\gdec,\varepsilon\lsdec\}$.}
  \par We will use \textbf{(T4)} on each tile.
  The exterior differential $d$ is a local operator \cite{weaver00}, i.e.:
  \begin{equation}
    \label{eq:glue_tiles_p5}
    df=df_{x_i}\quad\text{$\mu\on S_{x_i}$-a.e.}
  \end{equation}
  Thus:
  \begin{equation}
    \label{eq:glue_tiles_p6}
    \begin{split}
      \mu\left(
        \left\{
          x\in X: \wfnrm df(x).>\varepsilon
        \right\}
        \right) &\le
        \sum_i \mu\left(
        \left\{
          x\in S_{x_i}: \wfnrm df_{x_i}(x).>\varepsilon\gdec
        \right\}
      \right)\\ &\le
      \sum_i\varepsilon\lsdec\mu(S_{x_i})\\ &\le
      \varepsilon.
    \end{split}
  \end{equation}
  \par\noindent\texttt{Step 3: \textbf{(PosVar)} holds with
    $\alpha\vdec/2$ replacing $\alpha\vdec$.}
  \par We will use \textbf{(T5)} and the fact that $f$ and $f_{x_i}$ vanish on the
  boundary of each tile. Let $X\vdec=\bigcup_iS_{x_i,\vdec}$ so that $\mu(X\setminus
  X\vdec)\le\varepsilon$. Let $y\in S_{x_i,\vdec}$ and assume that the 
  $y\vdec$ corresponding to $f_{x_i}$ and $y$ is in $S_{x_i}$. Then:
  \begin{equation}
    \label{eq:glue_tiles_p7}
    \left|
      f(y) - f(y\vdec)
    \right| \ge
    \left|
      f_{x_i}(y) - f_{x_i}(y\vdec)
    \right| \ge
    \alpha\vdec d(y,y\vdec).
  \end{equation}
  If $y\vdec$ does not lie in $S_{x_i}$ we conclude that:
  \begin{equation}
    \label{eq:glue_tiles_p8}
    \left|f(y)\right|
    =
    \left|f_{x_i}(y)-f_{x_i}(y\vdec)\right|
    \ge\alpha\vdec d(y,y\vdec)\ge\alpha\vdec d(y,S_{x_i}^c).
  \end{equation}
  We then choose $\tilde{y}\in\partial S_{x_i}$ such that:
  \begin{equation}
    \label{eq:glue_tiles_p9}
    d(y,\tilde{y})\le\min\left(
      2d(y,S_{x_i}^c),d(y,y\vdec)
    \right),
  \end{equation}
  so that we have:
  \begin{equation}
    \label{eq:glue_tiles_p10}
    \begin{aligned}
      \left|f(y)-f(\tilde{y})\right|
      &=
      \left|f_{x_i}(y)\right|
      \ge\frac{\alpha\vdec}{2}d(y,\tilde{y})\\
      d(y,\tilde{y})&\le\varepsilon\lsdec\le\varepsilon.
    \end{aligned}
  \end{equation}
\end{proof}
\section{Blow-ups of differentiability spaces}
\label{sec:diff_blow}
\subsection{Blow-ups of differentiability spaces are differentiability
  spaces}
\label{subsec:diff_blow_diff}
\def\cubes{\mathscr{Q}}
\def\anndec{_{\text{\normalfont ann}}}
\def\dsdec{_{\text{\normalfont dens}}}
\def\todec{_{\text{\normalfont tan},1}}
\def\tddec{_{\text{\normalfont tan},2}}
\def\ttdec{_{\text{\normalfont tan},3}}
\def\ndfdec{_{\text{\normalfont ndiff}}}
\def\ctdec{_{\text{\normalfont cut}}}
\def\hudec{_{\text{\normalfont Hau}}}
\def\appdec{_{\text{\normalfont app}}}
\def\nvdec{_{n,\text{\normalfont var}}}
\def\linfty{\ell^\infty}
\pcreatenrm{wfxn}{\wform\mu_n.}{|}
\def\vseq#1.{\text{Seq}(#1)}
\def\dnet{\mathscr{N}}
\begin{thm}
  \label{thm:blow_up_diff}
  Let $(X,\mu)$ be a complete doubling metric measure space and assume
  that for $\mu$-a.e.~$p$ there is some $(Y,\nu,q)\in\tang(X,\mu,p)$
  which is not a differentiability space. Then $(X,\mu)$ is not a
  differentiability space.
\end{thm}
\begin{proof}
  In the following we will assume that $\mu(X)=1$. By
  Theorems~\ref{thm:indep_functions}, \ref{thm:glue_tiles} it suffices
  to show that at $\mu$-a.e.~point
  $p$ there is a sequence of
  ``bad tiles'' satisfying the assumption of
  Theorem~\ref{thm:glue_tiles}.
  \par\noindent\texttt{Step 1: Choice of Christ's cubes.}
  \par In $(X,\mu)$ we choose a set of Christ's dyadic cubes
  \cite{christ_cubes,hytonen_cubes} $\cubes$
  using the scales $\{2^{-n}\}_n$. The properties of $\cubes$ that we
  will use are:
  \begin{description}
  \item[(Cube1)] 
    {Cubes are open} and there are constants $C\anndec\ge 1$ and $\beta\anndec>0$
    such that for $\tau\in(0,1)$ and $Q\in\cubes$ the ``annulus''
    \begin{equation}
      \label{eq:blow_up_diff_p1}
      Q(\tau)=\left\{
        p\in Q: d(p,Q^c)\le\tau\diam Q
      \right\}
    \end{equation}
    satisfies:
    \begin{equation}
      \label{eq:blow_up_diff_p2}
      \mu\left(
        Q(\tau)
      \right) \le C\anndec\tau^{\beta\anndec}\mu(Q).
    \end{equation}
  \item[(Cube2)] There is a $k\anndec\ge1$ such that each ball
    $B(p,
    {r/2})$, where $r\in[2^{-n},2^{-n+1}]$, contains a cube $Q$ of generation
    $k\anndec+n$ and the measures $\mu(Q)$ and $\mu(B(p,r))$ are
    uniformly comparable.
  \end{description}
  In the following, we will only consider blow-ups at points $p\in X$ where
  the conclusion of Theorem~\ref{thm:blow_ups_are_blow_ups} holds.
  \par\noindent\texttt{Step 2: Uniformizing a bad function.}
  \par We now consider some $(Y,\nu,q)\in\tang(X,\mu,p)$ which is not a
  differentiability space. Note that $\nu$ is doubling with doubling
  constant bounded by $C^4_\mu$, $C_\mu$ being the doubling constant of $\mu$. Thus, the index of the
  module $\wder\nu.$ is bounded by $\lceil \log_2C^4_\mu\rceil$. We fix
  the parameter $\varepsilon\todec\in (0,1)$ and use the argument of
  Lemma \sync lem:der_alb_gap_flat_func.\ in \cite{deralb} to show that
  there are a Borel subset
  $S\ndfdec\subset Y$ with $\nu(S\ndfdec)>0$ and a Lipschitz function $f:Y\to\real$ such
  that:
  \begin{align}
    \label{eq:blow_up_diff_p3}
    \biglip f(x) &\in [1,2]\quad(\forall x\in S\ndfdec)\\
    \label{eq:blow_up_diff_p3_0}
    \wfynrm df(x). &\le\varepsilon\todec\quad(\forall x\in S\ndfdec).
  \end{align}
  We choose a point $\tilde{q}$ that is a Lebesgue density point $\tilde{q}$ of $S\ndfdec$, 
  and an approximate continuity point of $\biglip f$ and $\wfynrm
  df.$. By shifting the basepoint $\tilde{q}$ and rescaling
  $(Y,\nu,q)$ and $f$, we can assume that:
  \begin{description}
  \item[(BadF1)] $S\ndfdec\subset B_Y(q,2)$ and
    $\nu\left(B_Y(q,2)\setminus S\ndfdec\right)\le\varepsilon\todec$;
  \item[(BadF2)] The Lipschitz constant of the restriction of $f$ to
    $S\ndfdec$ is at most $3$;
  \item[(BadF3)] For $\nu$-a.e.~$x\in S\ndfdec$ there is a sequence $\vseq
    x.=\{x_n\}\subset S\ndfdec\setminus \{x\}$ converging to $x$ such
    that 
    {$f$ witnesses
    at $x$ a definite amount of variation at scale $d(x,x_n)$}:
    \begin{equation}
      \label{eq:blow_up_diff_p4}
      \left|
        f(x)-f(x_n)
      \right| \ge (1-\varepsilon\todec)d(x,x_n)>0.
    \end{equation}
  \end{description}
  We now use MacShane's Lemma to extend $f|S\ndfdec$ to a
  $3$-Lipschitz function $\tilde{f}:Y\to\real$. Note that
  \textbf{(BadF3)} and~(\ref{eq:blow_up_diff_p3_0}) remain valid
  replacing $f$ with $\tilde{f}$. In the following we will write $f$ for $\tilde{f}$.
  \par\noindent\texttt{Step 3: Truncating the function $f$.}
  \par Fix parameters $\tau\anndec\in(0,1)$ and
  $\tau\ctdec\in(0,\tau\anndec)$. We fix $N\in\natural$ and let
  \begin{equation}
    \label{eq:blow_up_diff_p5}
    \alpha\in[0,1)\cap\frac{1}{N^2}\zahlen
  \end{equation}
  we let:
  \begin{equation}
    \label{eq:blow_up_diff_p6}
    \psi_\alpha(\cdot)=d\left(\cdot,
      \frac{1}{N}\zahlen+\alpha
    \right),
  \end{equation}
  which is a $1$-Lipschitz function. Using the pigeonhole principle
   (see \cite[Lemma 4.1]{bate-diff} for details) we can find
  a value of $\alpha$ and a Borel set $S\vdec\subset S\ndfdec$ such that:
  \begin{equation}
    \label{eq:blow_up_diff_p7}
    \nu\left(
      S\ndfdec\setminus S\vdec
    \right) \le \frac{4}{N}\nu(S\ndfdec)\le
    {\frac{4}{N}\nu(B_Y(q,2))},
  \end{equation}
  such that for each $x\in S\ndfdec$ there is a sequence $\vseq
  x.\subset B_Y(q,2)\setminus\{x\}$ converging to $x$
  such that, 
  {similarly as in \textbf{(BadF3)}}, one has:
  \begin{equation}
    \label{eq:blow_up_diff_p8}
    \left|
      \psi_\alpha\circ f(x) - \psi_\alpha\circ f(x_n)
    \right| \ge (1-\varepsilon\todec)d(x,x_n)>0.
  \end{equation}
  Now note that the normalization
  \begin{equation}
    \label{eq:blow_up_diff_p8_1}
{ \int_{B_Y(q,1)}\left(1-d(p,x)\right)d\nu(x)=1}
  \end{equation}
  implies
  \begin{equation}
    \label{eq:blow_up_diff_p8_2}
{\nu\left(B_Y\left(q,\frac{1}{2}\right)\right)\le 2,}
  \end{equation}
 and thus~(\ref{eq:blow_up_diff_p7}) leads to:
 \begin{equation}
   \label{eq:blow_up_diff_p8_3}
{\nu\left(
      S\ndfdec\setminus S\vdec
    \right) \le \frac{8C_\mu^8}{N}.}
 \end{equation}
  Finally, note that:
  \begin{equation}
    \begin{aligned}
      \label{eq:blow_up_diff_p9}
      \|\psi_\alpha\circ f\|_\infty&\le\frac{1}{N}\\
      \wfynrm d(\psi_\alpha\circ f).&\le\wfynrm df..
    \end{aligned}
  \end{equation}
  We now fix a parameter $\varepsilon\tddec>\varepsilon\todec$
  and choose $N$ so that:
  \begin{description}
  \item[(BadC1)] $\nu(B(q,2)\setminus S\vdec)\le\varepsilon\tddec$ and
    $\|\psi_\alpha\circ f\|_\infty\le\tau\ctdec$;
  \item[(BadC2)] For each $x\in S\vdec$ there is a sequence $\vseq
    x.=\{x_n\}$ converging to $x$
    such that
    \begin{equation}
      \label{eq:blow_up_diff_p10}
      \left|
        \psi_\alpha\circ f(x) - \psi_\alpha\circ f(x_n)
      \right| \ge (1-\varepsilon\tddec)d(x,x_n)>0;
    \end{equation}
  \item[(BadC3)] For each $x\in S\vdec$ we have $\wfynrm d(\psi_\alpha\circ
    f)(x).\le\varepsilon\tddec$. 
  \end{description}
  In the following we will write $f$ for $\psi_\alpha\circ f$.
  \noindent\par\texttt{Step 4: Arranging convergence in $\linfty$.}
  \par We now choose a sequence of rescalings $\lambda_n\nearrow\infty$ 
  realizing $(Y,\nu,q)$, i.e.:
  \begin{equation}
    \label{eq:blow_up_diff_p11}
    (\lambda_nX=X_n,\mu_n,p)\xrightarrow{\text{mGH}}(Y,\nu,q);
  \end{equation}
  we can choose the convergence to take place in the container
  $(\linfty,0)$, and we will require that all basepoints
  map to $0$, but we will still distinguish them in the notation,
  i.e.~we
  will denote the
  basepoint of $X_n$ by $p_n$. We now fix the parameter
  $\varepsilon\ttdec>\varepsilon\tddec$. Then we can find
  $R\vdec\in(0,\tau\ctdec)$ and a compact set
  $K\hudec\subset B_Y(q,2)$ such that:
  \begin{description}
  \item[(Hau1)] $\nu(B_Y(q,2)\setminus K\hudec)\le\varepsilon\ttdec$
    and for each $x\in K\hudec$ one has \begin{equation}
\label{eq:blow_up_diff_p12_-1_break}{\wfynrm
    df(x).\le\varepsilon\ttdec};
\end{equation}
\item[(Hau2)] For each $x\in K\hudec$ there is an $x\vdec=x\vdec(x)$ satisfying:
    \begin{equation}
      \label{eq:blow_up_diff_p12}
      \begin{aligned}
        \left|
          f(x) - f(x\vdec)
        \right| &\ge (1-\varepsilon\ttdec)d(x,x\vdec)\\
        d(x,x\vdec)&\in(R\vdec,\tau\ctdec).
      \end{aligned}
    \end{equation}
  \end{description}
  Note that by MacShane's Lemma we can assume $f$ to be extended to a $3$-Lipschitz
  map $f:l^\infty\to\real$.
  \par\noindent\texttt{Step 5: Using the approximation scheme.}
  \par We fix the parameter $\varepsilon\ctdec\in(0,\varepsilon\ttdec
  R\vdec/2)$. Because of \textbf{(Hau1)} we can apply the
  Approximation Scheme Theorem~\ref{thm:approx} and find a function
  $\tilde{f}:\linfty\to\real$ which is $3$-Lipschitz, a compact set
  $K\appdec\subset K\hudec$, and an open set $U\supset K\appdec$ such
  that the following holds:
  \begin{description}
  \item[(App1)] $\|\tilde{f}-f\|_\infty\le\varepsilon\ctdec$ and
    $\nu(K\hudec\setminus K\appdec)\le\varepsilon\ctdec$;
  \item[(App2)] $U$ is an open set of $\linfty$ contained in the closed 
    $\varepsilon\ctdec$-neighbourhood $\bar
    B_{l^\infty}(K\appdec,\varepsilon\ctdec)$ of $K\appdec$;
  \item[(App3)] For each ball $B\subset U$ the restriction
    $\tilde{f}|B$ has Lipschitz constant $\le\varepsilon\ctdec+\varepsilon\ttdec$.
  \end{description}
  As $U$ is open and $\mu_n\xrightarrow{\text{w*}}\nu$ we
  have by the properties of the weak* topology and the choice of
  basepoints that for $n$
  sufficiently large:
  \begin{equation}
    \label{eq:blow_up_diff_p13}
{\int_{U\cap B_{\linfty}(0,1)}\left(1-d_{X_n}(p_n,x)\right)\,
    d\mu_n(x)\ge
    \int_{U\cap B_{\linfty}(0,1)}\left(1-d_{Y}(q,x)\right)\,
    d\nu(x)-\varepsilon\ctdec.}
  \end{equation}
  We thus obtain:
{\begin{equation}
    \label{eq:blow_up_diff_p14}
    \begin{split}
      \int_{U\cap B_{\linfty}(0,1)}\left(1-d_{X_n}(p_n,x)\right)\,
    d\mu_n(x)&\ge\int_{B_{\linfty}(0,1)}\left(1-d_{Y}(q,x)\right)\,
    d\nu(x)\\ &\mskip 8mu -\nu(B_Y(q,1)\setminus K\appdec)-\varepsilon\ctdec\\
       &\ge 1-(\varepsilon\ttdec+2\varepsilon\ctdec);
    \end{split}
  \end{equation}
  moreover, as on $B_{\linfty}(0,1/2)\cap X_n=B_{X_n}(p_n,1/2)$ we have
  that $1-d_{X_n}(p_n,\cdot)$ is at least $1/2$, we conclude that:
  \begin{equation}
    \label{eq:blow_up_diff_p14_0}
    \mu_n\left(
      B_{X_n}\left(
        p_n,\frac{1}{2}
      \right) \setminus U
    \right) \le 2(\varepsilon\ttdec+2\varepsilon\ctdec).
  \end{equation}}
      Now for $x\in U\cap B_{X_n}(p_n,1)$ we have:
  \begin{equation}
    \wfxnnrm d\tilde{f}(x).\le\varepsilon\ctdec+\varepsilon\ttdec.
  \end{equation}
  Moreover, we have:
  \begin{equation}
    \label{eq:blow_up_diff_p15}
    \|\tilde{f}\|_\infty\le\varepsilon\ctdec+\tau\ctdec,
  \end{equation}
  and for $x\in K\appdec$ we can find $x\vdec$
  satisfying~(\ref{eq:blow_up_diff_p12}). 
  \par\noindent\texttt{Step 6: Lifting the variation.}
  \par We now choose a parameter $\varepsilon\dsdec>0$ and a finite
  $\varepsilon\dsdec$-net $\dnet$ in $K\appdec\cap\bar B_Y(q,1)$. For
  each $x\in\dnet$ we can find $x\vdec(x)\in B_Y(q,2)$
  satisfying~(\ref{eq:blow_up_diff_p12}). We can thus construct the
  finite set:
  \begin{equation}
    \label{eq:blow_up_diff_p16}
    \dnet\vdec=\left\{
      x\vdec(x)
    \right\}_{x\in\dnet}
  \end{equation}
  and find a map $V:\dnet\to\dnet\vdec$ which associates $x\vdec(x)$
  to $x$. Using~(\ref{eq:gh_conv_3}), for $n$ sufficiently
  large, we can find a finite set
  \begin{equation}
    \label{eq:blow_up_diff_p17}
    \dnet_n\subset U\cap B_{X_n}(p_n,1+\varepsilon\dsdec)
  \end{equation} of the same cardinality as $\dnet$ and a bijection
  $J_n:\dnet\to\dnet_n$ such that:
  \begin{equation}
    \label{eq:blow_up_diff_p18}
    d(J_n(x),x)<\varepsilon\dsdec.
  \end{equation}
  Similarly, for $n$ sufficiently
  large, we can also find a finite set
  \begin{equation}
    \label{eq:blow_up_diff_p19}
    \dnet\nvdec\subset U\cap B_{X_n}(p_n,2+\varepsilon\dsdec)
  \end{equation} of the same cardinality as $\dnet\vdec$ and a bijection
  $J\nvdec:\dnet\vdec\to\dnet\nvdec$ such that:
  \begin{equation}
    \label{eq:blow_up_diff_p20}
    d(J\nvdec(x),x)<\varepsilon\dsdec.
  \end{equation}
  We finally let $V_n=J\nvdec\circ V\circ J_n^{-1}$. For $y\in U\cap
  B_{X_n}(p_n,1)$ we can find $x\in\dnet$ such that:
  \begin{equation}
    \label{eq:blow_up_diff_p21}
    d(y,x)=O(\varepsilon\ctdec,\varepsilon\dsdec);
  \end{equation}
  using that $\tilde{f}$ is $3$-Lipschitz we obtain:
  \begin{equation}
    \label{eq:blow_up_diff_p22}
    \left|
      \tilde{f}(y)
      - \tilde{f}(V_n(x))
    \right| = \left|
      \tilde{f}(J_n^{-1}(x)) - \tilde{f}(V \circ J_n^{-1}(x))
    \right| + O(\varepsilon\ctdec,\varepsilon\dsdec);
  \end{equation}
  using the properties of $V$ we also conclude that:
  \begin{equation}
    \label{eq:blow_up_diff_p23}
    d(y,V_n(x))=d(J_n^{-1}(x),V\circ J_n^{-1}(x))+O(\varepsilon\ctdec,\varepsilon\dsdec).
  \end{equation}
  Note that the constants hidden in the $O(\cdot)$ notation
  in~(\ref{eq:blow_up_diff_p21})--(\ref{eq:blow_up_diff_p23}) do not
  depend on $n$. As the parameters $\varepsilon\ctdec$ and
  $\varepsilon\dsdec$ are chosen after $R\vdec$ and
  $\varepsilon\ttdec$, one can choose them sufficiently small so that:
  \begin{equation}
    \label{eq:blow_up_diff_p24}
    \begin{aligned}
      \left|\tilde{f}(y)-\tilde{f}(V_n(x))\right|&\ge(1-\varepsilon\ttdec)d(y,V_n(x))\\
      d(y,V_n(x))&\in(R\vdec/2,2\tau\ctdec).
    \end{aligned}
  \end{equation}
  \par\noindent\texttt{Step 7: Constructing the tiles.}
  \par Let $Q\in\cubes$ be a dyadic cube of generation
  $g_n=\lfloor\log_2\lambda_n\rfloor+k\anndec$ contained in
  {$B_{X_n}(p_n,1/2)$}. For the moment we will compute distances using the
  rescaled metric. The cube $Q$ will be used to
  construct the tile. 
  {Specifically, recall that $Q$ is open and
  consider a parameter $\tau\anndec\in(0,1/8)$ to be chosen later. The
  set $\tilde Q$ that we will use for the tile will be the closure in
  $X_n$ of $Q\setminus Q(\tau\anndec)$.}
  Now \textbf{(Cube1)} and \textbf{(Cube2)} imply that there is a
  uniform constant $C$ such that:
  \begin{equation}
    \label{eq:blow_up_diff_p25}
    \begin{aligned}
      {\diam \tilde Q \approx_C} \diam Q&\approx_C 2^{-k\anndec}\\
      {\mu_n(\tilde Q) \approx_C} \mu_n(Q)&\gtrsim_C\mu_n(B_{X_n}(p_n,1));
    \end{aligned}
  \end{equation}
  this will give \textbf{(T1)} and \textbf{(T2)}.
  Consider $Q(\tau\anndec)$ and a
  $\frac{C2^{k\anndec}}{\tau\anndec}$-Lipschitz function $\psi_Q$
  which takes values in $[0,1]$ and such
  that:
  \begin{equation}
    \label{eq:blow_up_diff_p26}
    \begin{aligned}
      \psi_Q &=
      \begin{cases}
        1&\text{on $Q\setminus Q(
          {2}\tau\anndec)$}\\
        0&\text{on $\tilde Q^c$}      
      \end{cases}\\
      \left|
        \psi_Q(x)
      \right|&\le\frac{C2^{k\anndec}}{\tau\anndec} d(x,\tilde Q^c).
    \end{aligned}
  \end{equation}
  Let $h=\psi_Q\tilde{f}$; using equation~(\ref{eq:blow_up_diff_p15})
  we conclude that $h$ has Lipschitz constant at most
  \begin{equation}
    \label{eq:blow_up_diff_p27}
    3+\frac{C2^{k\anndec}}{\tau\anndec}(\tau\ctdec+\varepsilon\ctdec).
  \end{equation}
  As the parameters $\tau\ctdec$ and $\varepsilon\ctdec$ are chosen
  after $k\anndec$ and $\tau\anndec$ have been determined, we can
  choose them small enough to ensure that $h$ is $4$-Lipschitz. We
  then also have:
  \begin{equation}
    \label{eq:blow_up_diff_p_ins_28-1}
    \left|
      h(x)
    \right|\le
    \frac{C2^{k\anndec}}
    {\tau\anndec} (\tau\ctdec+\varepsilon\ctdec)d(x,\tilde Q^c)\\
    \le 4d(x,\tilde Q^c).
  \end{equation}
  The function $h$ will be the function that we use to construct the
  tile. Now \textbf{(T3)} follows from~(\ref{eq:blow_up_diff_p27}), (\ref{eq:blow_up_diff_p_ins_28-1}).
  We now let:
  \begin{equation}
    \label{eq:blow_up_diff_p28}
    {Q\ndfdec=Q\setminus Q(2\tau\anndec+2\tau\ctdec)\cap U \subset\tilde
    Q \subset B_{X_n}(p_n,1/2)};
  \end{equation}
  we then have, using
  \textbf{(Cube1)} and minding~(\ref{eq:blow_up_diff_p14_0}):
  {\begin{equation}
    \label{eq:blow_up_diff_p29}
    \begin{split}
      \mu_n(\tilde Q\setminus Q\ndfdec) &\le
      \mu_n(Q(2\tau\anndec+2\tau\ctdec)) +
      \mu_n(B_{X_n}(p_n,1/2)\setminus U) \\
      &\le
      C\anndec(2\tau\anndec+2\tau\ctdec)^{\beta\anndec}\mu_n(Q)
      + \mu_n(B_{X_n}(p_n,1/2)\setminus U)\\
      &\le
      C\anndec\,
      C(2\tau\anndec+2\tau\ctdec)^{\beta\anndec}\mu_n(\tilde Q)
      + 2(\varepsilon\ttdec+2\varepsilon\ctdec).
    \end{split}
  \end{equation}}
  As the parameters on the right hand side
  of~(\ref{eq:blow_up_diff_p29}) are chosen after \texttt{Step 1},
  given a fixed $\varepsilon\lsdec>0$ we chan choose those parameters
  so that:
  \begin{equation}
    \label{eq:blow_up_diff_p30}
    \mu_n(\tilde Q\setminus Q\ndfdec)\le\varepsilon\lsdec\mu_n(\tilde Q). 
  \end{equation}
  Now, if $x\in Q\ndfdec$, the function $h$ is
  $(\varepsilon\ctdec+\varepsilon\ttdec)$-Lipschitz in a
  neighbourhood of $x$ and hence:
  \begin{equation}
    \label{eq:blow_up_diff_p31}
     \wfxnnrm dh(x).\le(\varepsilon\ctdec+\varepsilon\ttdec).
   \end{equation}
   Given $\varepsilon\gdec$ we can choose $\varepsilon\ctdec$ and
   $\varepsilon\ttdec$ so that their sum is $\le\varepsilon\gdec$;
   this gives \textbf{(T4)}.
   Also, given $x\in Q\ndfdec$ we can find
   by~(\ref{eq:blow_up_diff_p24}) a point $
   {y\in Q\setminus
   Q(\tau\anndec)\subset\tilde Q}$ such that:
   \begin{equation}
     \label{eq:blow_up_diff_p32}
     \begin{aligned}
       \left|
         h(x) - h(y)
       \right| &\ge(1-2\varepsilon\ttdec)d(x,y)>0\\
       d(x,y)&\le2\tau\ctdec.
     \end{aligned}
   \end{equation}
   Choosing $\tau\ctdec$ sufficiently small we can ensure that
   \begin{equation}
     \label{eq:blow_up_diff_ins_-1_p33}
     d(x,y)\le\varepsilon\lsdec.
   \end{equation}
   Finally, choosing $\varepsilon\ttdec$
   sufficiently small we can also ensure that:
   \begin{equation}
     \label{eq:blow_up_diff_p33}
     1-2\varepsilon\ttdec\ge\frac{2}{3};
   \end{equation}
   thus (\ref{eq:blow_up_diff_p32})--(\ref{eq:blow_up_diff_p33}) give \textbf{(T5)}.
   Rescaling $h$ back to $X$ (i.e.~$h\mapsto \lambda_n^{-1}h$), we
   conclude that $(\tilde Q,h)$ is an:
   \begin{equation}
     \label{eq:blow_up_diff_p34}
     \left[
       \max(C,5), \frac{2}{3}, \varepsilon\gdec, \varepsilon\lsdec, \lambda_n^{-1}
     \right]\text{-tile}
   \end{equation}
   at $p$.
\end{proof}
\subsection{Independence of the $p$-weak gradient on $p$}
\label{subsec:weak_grad_indep}
\def\slda{\text{\normalfont Slide}} 
\begin{defn}
  \label{defn:slide}
  Let $I\subset\real$ be a nondegenerate closed interval and
  $A_I:\real\to\real$ the unique orientation-preserving affine map
  which maps $[0,1]$ onto $I$. For $\varepsilon>0$ let:
  \begin{equation}
    \label{eq:slide_1}
    \slda(I,\varepsilon) = \left\{\gamma\in\curves(X): \forall
      t\in[0,\varepsilon]\quad t+I\subset\dom\gamma
    \right\};
  \end{equation}
  let:
  \begin{equation}
    \label{eq:slide_2}
    \slda_{I,\varepsilon}:\slda(I,\varepsilon)\times[0,\varepsilon]\to\curves(X,[0,1])
  \end{equation}
  be the map such that:
  \begin{equation}
    \label{eq:slide_3}
    \slda_{I,\varepsilon}(\gamma,t)(s)=\gamma(A_{I}(s)+t).
  \end{equation}
\end{defn}
\begin{lem}[Test plan associated with an Alberti representation]
  \label{lem:test_plan_arep}
  Let $\albrep.=[Q,1]$ be an Alberti representation of a measure $\nu$
  such that for some $C_\nu>0$ one has $\nu\le C_\nu\mu$. Assume that
  for some closed interval $I$, $\varepsilon>0$ and $C_0>0$ the
  measure $Q$ is concentrated on the set of $C_0$-Lipschitz curves in
  $\slda(I,\varepsilon)$. Assuming that $Q$ is a probability measure,
  we can associate to $\albrep.$ a probability $\pi$ on
  $\curves(X;[0,1])$ by:
  \begin{equation}
    \label{eq:test_plan_arep_s1}
    \pi=\slda_{I,\varepsilon,\#}\left(
      Q\times\frac{1}{\varepsilon}\lebmeas.\on[0,\varepsilon]
    \right).
  \end{equation}
  Then for any $q\in[1,\infty)$, $\pi$ defines a $q$-test plan.
\end{lem}
\begin{proof}
  Note that for $Q$-a.e.~$\gamma$ one has that $\metdiff\gamma\le C_0$
  holds $\lebmeas.\on\dom\gamma$-a.e. Now the derivative of $A_I$ is
  $\lebmeas.(I)$ and so for $\pi$-a.e.~$\gamma$ one has that
  $\metdiff\gamma\le C_0\lebmeas.(I)$ holds $\lebmeas.\on\dom\gamma$-a.e.;
  one thus gets:
  \begin{equation}
    \label{eq:test_plan_arep_p1}
    \int d\pi(\gamma)\int_0^1\left(
      \metdiff\gamma(t)
    \right)^q\,dt\le \left(
      C_0\lebmeas.(I)
    \right)^q,
  \end{equation}
  which gives~(\ref{eq:test_plan_1}).
  Let $\varphi$ be a nonnegative continuous function of $X$; then:
  \begin{equation}
    \begin{split}
      \label{eq:test_plan_arep_p2}
      \int \varphi\,d\evala_{t\#}\pi &=
      \int \varphi(\gamma(t))\,d\pi(\gamma) =
      \int dQ(\gamma) \frac{1}{\varepsilon} \int_0^\varepsilon
      \varphi\left(
        \gamma(A_I(t)+s)
      \right)\,ds \\ &\le
      \frac{1}{\varepsilon}\int\varphi\,d\nu\le
      \frac{C_\nu}{\varepsilon}\int\varphi\,d\mu,
    \end{split}
  \end{equation}
  which establishes~(\ref{test_plan_2}).
\end{proof}
\begin{defn}[Regular Alberti representation]
  \label{defn:reg_alb}
  An Alberti representation $[Q,1]$ is \textbf{regular} if $Q$ is concentrated
  on the set of unit-speed geodesic lines of $X$. 
  {Here we think of
  unit-speed geodesic lines as maps $\gamma:\real\to X$, and so they
  have infinite length.}
\end{defn}
\begin{lem}
  \label{lem:reg_alb_est}
  Let $[Q,1]$ be a regular Alberti-representation of $\mu$ and
  $(x,r)\in X\times(0,\infty)$; then:
  \begin{equation}
    \label{eq:reg_alb_est_p1}
    Q\left(
      \left\{
        \gamma: \gamma^{-1}\left(
          B(x,r)
        \right) \ne\emptyset
      \right\}
    \right) \le \frac{\mu\left(B(x,2r)\right)}{r}.
  \end{equation}
\end{lem}
\begin{proof}
  It suffices to observe that if $\gamma$ is a unit-speed geodesic
  line in $X$ then $\gamma^{-1}\left(B(x,r)\right)\ne\emptyset$
  implies:
  \begin{equation}
    \label{eq:reg_alb_est_s1}
    \lebmeas. \left(
      \gamma^{-1}\left(B(x,2r)\right)
    \right)\ge r.
  \end{equation}
\end{proof}
\begin{defn}
  \label{defn:regular_diff}
  A differentiability space $(X,\mu)$ is \textbf{regular} if:
  \begin{description}
  \item[(Reg1)] The measure $\mu$ is doubling and there is a unique
    differentiability chart $(X,\Phi=\{\varphi_i\}_{i=1}^n)$ such that
    $\Phi:X\to\real^n$ is $1$-Lipschitz;
  \item[(Reg2)] The local norm $\wfnrmname$ is constant;
  \item[(Reg3)] For each vector $v\in\wder\mu.$ (note that we can
    canonically identify a vector in the measurable tangent bundle
    with a derivation, as the chart is global) with $\wdnrm v.=1$
    there is an Alberti representation $[Q_v,1]$ of $\mu$ where $Q_v$
    is concentrated on the set of unit-speed lines $\gamma$ in $X$
    satisfying:
    \begin{equation}
      \label{eq:regular_diff_1}
      (\Phi\circ\gamma)'=v.
    \end{equation}
  \end{description}  
\end{defn}
\begin{lem}[Iterated blow-ups are regular]
  \label{lem:regular_diff}
  Let $(X,\mu)$ be a differentiability space. Then for $\mu$-a.e.~$x$
  there is an integer $N(x)$ such that, whenever
  $(Y,\nu,y)\in\tang(X,\mu,x)$, up to passing to at most $N(x)$
  iterated blow-ups, i.e.~up to replacing $(Y,\nu,y)$ with
  $(Y_l,\nu_l,y_l)$ where:
  \begin{equation}
    \label{eq:regular_diff_s1}
    \begin{aligned}
      (Y_i,\nu_i,y_i)&\in\tang(Y_{i-1},\nu_{i-1},y_{i-1})\quad 1\le
      i\le l\le N(x)\\
      (Y_0,\nu_0,y_0)&=(Y,\nu,y),
    \end{aligned}
  \end{equation}
  one can assume that $(Y,\nu,y)$ is a regular differentiability
  space. The integer $N(x)$ satisfies:
  \begin{equation}
    \label{eq:regular_diff_s2}
    N(x)\le\lfloor
    \log_2C_\mu(x)
    \rfloor,
  \end{equation}
  where $C_\mu(x)$ is the asymptotically doubling constant of $\mu$ at
  $X$. In particular ,if $X$ has finite Assoaud dimension $N$, then
  \begin{equation}
    \label{eq:regular_diff_s3}
    N(x)\le N.
  \end{equation}
\end{lem}
\begin{proof}
  At $\mu$-a.e.~$x$ the conclusion of
  Theorem~\ref{thm:blow_ups_are_blow_ups} holds, and each
  $(Y,\nu,y)\in\tang(X,\mu,x)$ is a differentiability space. Now by
  \textbf{(IndBound)} in Theorem~\ref{thm:der_alb_corr} the index of $\wder\nu.$ is uniformly bounded by
  some $N(x)$ which satisfies~(\ref{eq:regular_diff_s2}),
  (\ref{eq:regular_diff_s3}). Thus in passing to iterated blow-ups as
  in~(\ref{eq:regular_diff_s1}), the differentiability dimension can
  increase at most $N(x)$ times. If $l$ is the smallest integer such
  that in passing from $(Y_{l-1},\nu_{l-1},y_{l-1})$ to
  $(Y_l,\nu_l,y_l)$ the differentiability dimension does not increase,
  then $(Y_l,\nu_l,y_l)$ is regular by Theorem~\ref{thm:weaver_blowup}
  (compare also \cite{cks_metric_diff}).
\end{proof}
\def\condec{_{\text{\normalfont cont}}}
\def\shdec{_{\text{\normalfont short}}}
\begin{thm}[The $p$-weak gradient does not depend on $p$]
  \label{thm:weak_grad_indep}
  Let $(X,\mu)$ be a regular differentiability space and $f$ a
  Lipschitz function on $X$; then for any $p\in(1,\infty)$:
  \begin{equation}
    \label{eq:weak_grad_indep_s1}
    \pwgradnrm f.=\biglip f\quad\text{\normalfont $\mu$-a.e.}
  \end{equation}
\end{thm}
\begin{proof}
  \noindent\texttt{Step 1: Uniformization.}
  \par Let $g$ be a $p$-weak upper gradient of $f$ and let $A\condec$
  denote the set of differentiability points of $f$ and of
  approximate continuity points of $df$, $g$ and
  $\biglip f$. We fix $x\in A\condec$ and choose the parameter
  $\varepsilon\condec>0$. Up to rescaling $f$ and $g$ we can assume
  that:
  \begin{equation}
    \label{eq:weak_grad_indep_p1}
    \wfnrm df(x).=\biglip f(x)=1.
  \end{equation}
  We let $v\in\wder\mu.$ be a unit vector where $df(x)$ attains the
  norm and let:
  \begin{equation}
    \label{eq:weak_grad_indep_p2}
    A_x = \biggl\{
      y\in A\condec: \text{$\left|
        df(y) - df(x)
      \right|\le\varepsilon\condec$ and
      $
      \left|
        g(y)-g(x)
      \right| \le\varepsilon\condec
      $
    }
    \biggr\}.    
  \end{equation}
  Note that for $r$ sufficiently small we can assume:
  \begin{equation}
    \label{eq:weak_grad_indep_pins_+2}
    \mu\left(
      B(x,r)\setminus A_x
    \right)\le\varepsilon\condec\mu\left(
      B(x,r)
    \right).
  \end{equation}
  \par\noindent\texttt{Step 2: Construction of a $q$-test plan.}
  \par Let $q$ be the dual exponent of $p$. Let $[Q,1]$ be an Alberti
  representation of $\mu$ as in the definition of a regular
  differentiability space where we use the unit vector $v$. Let
  \begin{equation}
    \label{eq:weak_grad_indep_p3}
    \Gamma_{x,r}=\biggl\{
    \text{
      $\gamma$ is a unit speed line with $\gamma^{-1}(B(x,r))\ne\emptyset$
    }
    \biggr\};
  \end{equation}
  letting $Q_{x,r}=Q\on\Gamma_{x,r}$ we have by
  Lemma~\ref{lem:reg_alb_est}:
  \begin{equation}
    \label{eq:weak_grad_indep_p4}
    \|Q_{x,r}\|\le\frac{\mu\left(B(x,2r)\right)}{r}.
  \end{equation}
  Let $\mu_{x,r}$ be the measure corresponding to the Alberti
  representation $[Q_{x,r},1]$ and note that:
  \begin{equation}
    \label{eq:weak_grad_indep_p5}
    \begin{aligned}
      \mu_{x,r}&\le\mu\\
      \frac{d\mu_{x,r}}{d\mu}&=1\quad\text{on $B(x,r)$.}
    \end{aligned}
  \end{equation}
  We now let:
  \begin{equation}
    \label{eq:weak_grad_indep_p6}
    \Gamma\shdec=\biggl\{
    \text{
      $\gamma$ is a unit speed line with $\gamma^{-1}(A_x\cap
      B(x,r))\le\varepsilon\shdec r$
    }
    \biggr\};
  \end{equation}
  we then obtain the estimates:
  \begin{equation}
    \label{eq:weak_grad_indep_p7}
    \mu_{x,r}\left(A_x\cap B(x,r)\right)\ge
    (1-\varepsilon\condec)\mu\left(B(x,r)\right)
  \end{equation}
  \begin{equation}
    \label{eq:weak_grad_indep_p8}
    \begin{split}
      \mu_{x,r}\left(A_x\cap
        B(x,r)\right)&\le\int_{\Gamma\shdec}\lebmeas.\left(
        \gamma^{-1}(A_x\cap B(x,r))
      \right)\,dQ_{x,r}(\gamma)\\
      &+ \int_{\Gamma^{c}\shdec}\lebmeas.\left(
        \gamma^{-1}(A_x\cap B(x,r))
      \right)\,dQ_{x,r}(\gamma)\\
      &\le\varepsilon\shdec r\|Q_{x,r}\|\\
      &+\int_{\Gamma^{c}\shdec}\lebmeas.\left(
        \gamma^{-1}(A_x\cap B(x,r))
      \right)\,dQ_{x,r}(\gamma)\\
      &\le\varepsilon\shdec \mu(B(x,2r))\\
      &+\int_{\Gamma^{c}\shdec}\lebmeas.\left(
        \gamma^{-1}(A_x\cap B(x,r))
      \right)\,dQ_{x,r}(\gamma);
    \end{split}
  \end{equation}
  combining~(\ref{eq:weak_grad_indep_p7}),
  (\ref{eq:weak_grad_indep_p8}) we conclude that for
  $\varepsilon\shdec$ sufficiently small:
  \begin{equation}
    \label{eq:weak_grad_indep_p9}
    \int_{\Gamma^{c}\shdec}\lebmeas.\left(
      \gamma^{-1}(A_x\cap B(x,r))
    \right)\,dQ_{x,r}(\gamma)>0.
  \end{equation}
  Using a measurable selection principle (see
  \cite[Chap.~18]{kechris_desc}) 
  we can find an Alberti representation $[Q_{A_x},1]$ of a
  measure $\nu\ll\mu$ such that:
  \begin{description}
  \item[(Test1)] $Q_{A_x}$ is a finite Radon measure,
    $\frac{d\nu}{d\mu}\le 1$;
  \item[(Test2)] For $Q_{A_x}$-a.e.~$\gamma$, letting $I=[0,3r]$ we
    have for each $t\in[0,r]$:
    \begin{equation}
      \label{eq:weak_grad_indep_p10}
      \begin{aligned}
        I+t&\subset\dom\gamma\\
        \lebmeas.\left(
          \gamma^{-1}\left(
            A_x\cap B(x,r)
          \right) \cap (I+t)
        \right)&>0;
      \end{aligned}
    \end{equation}
  \item[(Test3)] $Q_{A_x}$-a.e.~$\gamma$ is a unit-speed geodesic
    segment of length at most $6r$.
  \end{description}
  Using Lemma~\ref{lem:test_plan_arep} (and setting $\varepsilon=r$) we can associate to
  $[Q_{A_x}/\|Q_{A_x}\|,1]$ a $q$-test plan $\pi$.
  \par\noindent\texttt{Step 3: Applying Lebesgue's differentiation
    along curves.}
  \par As $g\in L^p(\mu)$, conditions
  \textbf{(Test1)}--\textbf{(Test3)} imply
  that for $\pi$-a.e.~$\gamma$ there is a nondegenerate interval
  $J_\gamma\subset\dom\gamma$ such that $g\circ\gamma\in L^1(\lebmeas.\on J_\gamma)$
  and $J_\gamma$ meets $\gamma^{-1}(A_x)$ in positive Lebesgue
  measure. Moreover, for $\pi$-a.e.~$\gamma$ for each $(a,b)\subset
  J_\gamma$ we also have that:
  \begin{equation}
    \label{eq:weak_grad_indep_p11}
    \left|
      f\circ\gamma(b) - f\circ\gamma(b) 
    \right| \le\int_a^bg\circ\gamma\,d\lebmeas.;
  \end{equation}
  if $t_0\in J_\gamma$ is an interior point of $J_\gamma$, a differentiability point of $f\circ\gamma$,
  and a Lebesgue point of $\gamma^{-1}(A_x)$, applying Lebesgue's
  differentiation at $t_0$ we
  obtain:
  \begin{equation}
    \label{eq:weak_grad_indep_p12}
    g(\gamma(t_0))\ge\langle df(\gamma(t_0)),v\rangle,
  \end{equation}
  which implies
  \begin{equation}
    \label{eq:weak_grad_indep_p13}
    g(x)\ge 1-2\varepsilon\condec.
  \end{equation}
\end{proof}
\begin{rem}
  \label{rem:bv}
  Theorem~\ref{thm:weak_grad_indep} has a counterpart in the category
  of functions of bounded variations; i.e.~one can show that in a
  regular differentiability space the \textbf{total variation measure}
  of a Lipschitz function $f$ coincides with $\biglip f\cdot\mu$; we omit the details
  because they are mainly technical and do not add much mathematical
  substance to the paper.
\end{rem}
\bibliographystyle{alpha}
\bibliography{dim_blow_biblio}
\end{document}